\documentclass{ws-ijm}
\usepackage{cite}
\usepackage{amsmath}
\usepackage[mathcal]{euscript}
\usepackage{tikz-cd}

\let\qedhere\null

\def\C{\mathbb{C}}
\def\N{\mathbb{N}}
\def\id{\mathrm{id}}

\begin{document}

\markboth{A. Kuzmin, V. Ostrovskyi, D. Proskurin, M. Weber and R. Yakymiv}{On $q$-tensor products of Cuntz algebras}

\catchline{}{}{}{}{}

\title{On $q$-tensor products of Cuntz algebras}

\author{Alexey Kuzmin}

\address{Department of Mathematical Sciences,\\ Chalmers University of Technology and University of Gothenburg,\\ Gothenburg, Sweden, vagnard.k@gmail.com}

\author{Vasyl Ostrovskyi}

\address{Institute of Mathematics, NAS of Ukraine,\\ Kyiv, Ukraine, vo@imath.kiev.ua}

\author{Danylo Proskurin}

\address{Faculty of Computer Sciences and Cybernetics,\\ Kyiv National Taras Shevchenko University,\\ Kyiv, Ukraine, prosk@univ.kiev.ua}

\author{Moritz Weber}

\address{Faculty of Mathematics,\\ Saarland University, Saarbr\"ucken, Germany, weber@math.uni-sb.de}

\author{Roman Yakymiv}  

\address{Faculty of Computer Sciences and Cybernetics,\\ Kyiv National Taras Shevchenko University,\\ Kyiv, Ukraine, yakymiv@univ.kiev.ua}

\maketitle

\begin{dedication}
{\it To 75-th birthday of our teacher Yurii S. Samoilenko}
\end{dedication}

\begin{abstract}
We consider the $C^*$-algebra $\mathcal{E}_{n,m}^q$, which is a $q$-twist of two Cuntz-Toeplitz algebras. For the case $|q|<1$, we give an explicit formula which untwists the $q$-deformation showing that the isomorphism class of $\mathcal{E}_{n,m}^q$ does not depend on $q$. For the case $|q|=1$, we give an explicit description of all ideals in $\mathcal{E}_{n,m}^q$. In particular, we show that $\mathcal{E}_{n,m}^q$ contains a unique largest ideal $\mathcal{M}_q$. We identify $\mathcal{E}_{n,m}^q / \mathcal{M}_q$ with the Rieffel deformation of  $\mathcal{O}_n \otimes \mathcal{O}_m$ and use a K-theoretical argument to show that the isomorphism class does not depend on $q$. The latter result holds true in a more general setting of multiparameter deformations.
\end{abstract}



\section{Introduction}\label{intro}

 Since the early 80's, a wide study of non-classical models of mathematical physics, quantum group theory and noncommutative probability (see e.g., \cite{BoSpe2,fiv,green,mac,MPe,zag}) gave rise to a number of papers on operator algebras generated by various deformed commutation relations \cite{BoSpe,Klimek,mar}, a prominent example being the  irrational rotation algebra, also called the non-commutative torus \cite{rieff}. A major question for such objects is whether deformations exist and how they relate to the original object.

 Other objects of studies, closely related to the ones mentioned above, are $C^*$-algebras generated by isometries, such as Cuntz algebras, extensions of non-commutative tori and their multi-parameter generalizations, see, \cite{coburn, cun, jeu_pinto, Popescu, prolett}. The problems of classification of representations, existence of faithful (universal) representation, as well as the a study of structure of corresponding $C^*$-algebra and dependence of $C^*$-isomorphism classes on parameter of the deformations  are among the central ones in this area. 
 
In our article, we work with a certain class of $C^*$-algebras generated by isometries subject to deformed commutation relations. We study the structure of these $C^*$\nobreakdash-algebras, present their faithful representations and show that some of the algebras we deal with, are independent of the deformation parameter. of deformation. The relation with Rieffel deformation of tensor products, see e.g \cite{Kasprzak_rieffel}, has been significant for our studies.

 \subsection{Context: Wick algebras}

 Let us put our work into a broader context. 
 A general approach to the study of such deformed commutation relations has been provided by the framework of quadratic $*$-algebras allowing Wick ordering (Wick algebras), see \cite{jsw}. It includes, among others, deformations of canonical commutation relations of quantum mechanics, some quantum groups and quantum homogeneous spaces, see e.g., \cite{gisselson,Klimyk,vaksman1,woronowicz}.
On~the other hand, one can consider Wick algebras as deformations of Cuntz-Toeplitz algebras, see \cite{cun,dn,jsw}.

For $\{T_{ij}^{kl},\ i,j,k,l=\overline{1,d}\}\subset\mathbb C$, $T_{ij}^{kl}=\overline{T}_{ji}^{lk}$, the Wick algebra  $W(T)$
is the $*$\nobreakdash-algebra generated by elements $a_j$, $a_j^*$, $j=\overline{1,d}$ subject to the relations
\[
a_i^*a_j=\delta_{ij}\mathbf 1+\sum_{k,l=1}^d T_{ij}^{kl} a_l a_k^*.
\]
It depends \cite{jsw}  on the so called operator of coefficients $T$, 
 given as follows. Let $\mathcal{H}=\mathbb C^d$ and $e_1$, \dots, $e_d$ be the standard orthonormal basis, then
\[
T\colon \mathcal{H}^{\otimes 2}\rightarrow \mathcal{H}^{\otimes 2},\quad
T e_k\otimes e_l=\sum_{i,j=1}^d T_{ik}^{lj} e_i\otimes e_j.
\]

It is a non-trivial and central problem in the theory of Wick algebras to determine whether a Fock representation $\pi_{F, T}$ of a Wick algebra exists, see \cite{BoSpe,jps,jsw}  for some sufficient conditions: for instance, it exists, if $T$ is braided, i.e., $(\mathbf 1\otimes T)(T\otimes\mathbf 1)
(\mathbf 1\otimes T)=(T\otimes\mathbf 1)
(\mathbf 1\otimes T)(\mathbf 1\otimes T)$, and if $\lVert T\rVert\leq 1$; moreover, if $\lVert T\rVert<1$ then $\pi_{F,T}$ is a faithful representation of $W(T)$. 

Another important question concerns the stability of isomorphism classes of the universal $C^*$-envelope\footnote{Recall that given a $*$-algebra $A$, we denote by $\mathcal A = C^*(A)$ its universal $C^*$-algebra, if it exists, i.e. if the set $\mathsf{Rep}\,  A$ of bounded $*$-representations of $ A$ is non-empty and 
\[
\sup_{\pi\in\mathsf{Rep}\, A}\|\pi( a)\|<\infty
\] 
for any $a\in A$. The universal $C^*$-algebra $\mathcal A$ is determined by the  universal property: there exists a $*$-homomorphism $\theta \colon A \to \mathcal A$ such that for any $C^*$-algebra $\mathcal B$ and $*$-homomorphism $\beta \colon A\to \mathcal B$, there exists a unique $*$-homomorphism $\tilde \beta\colon \mathcal A\to\mathcal B$, such that $\beta = \tilde \beta \circ \theta$.} 
$\mathcal{W}(T) = C^*(W(T))$. It was conjectured in  \cite{jsw2}:
\begin{conjecture}\label{q_stab}
If  $T$ is  self-adjoint, braided and $||T|| < 1$, then $\mathcal{W}(T) \simeq \mathcal{W}(0)$.
\end{conjecture}

In particular, the authors of \cite{jsw2} have shown that the conjecture holds for the case $||T|| < \sqrt{2} - 1$, for more results on the subject see \cite{dn}, \cite{nk}.

In the case $T=0$ and $d=\dim \mathcal{H} =1$, the Wick algebra $W(0)$ is generated by a single isometry $s$, its  universal $C^*$-algebra exists and is isomorphic to the $C^*$-algebra generated by the unilateral shift, and the Fock representation is faithful. The ideal $\mathcal I$ in $\mathcal E$, generated by $\mathbf 1 - ss^*$ is isomorphic to the algebra of compact operators and $\mathcal E/\mathcal I\simeq C(S^1)$, see \cite{coburn}. When $d\ge 2$, the enveloping universal $C^*$-algebra   exists and it is called the Cuntz-Toeplitz algebra $\mathcal O_d^{(0)}$. It is isomorphic to $C^*(\pi_{F,d}(W(0)))$, so the Fock representation of $\mathcal O_d^{(0)}$ is faithful, see \cite{cun}. Furthermore, the ideal $\mathcal I$ generated by $1-\sum_{j=1}^d s_js_j^*$ is the unique largest ideal in $\mathcal O_d^{(0)}$. It is isomorphic to the algebra of compact operators on $\mathcal F_d$. The quotient $\mathcal O_d^{(0)}/\mathcal I$ is called the Cuntz algebra $\mathcal O_d$. It~is nuclear (as well as $\mathcal O_d^{(0)}$),  simple and purely infinite, see \cite{cun} for more details.

\subsection{Our objects of interest: the $C^*$-algebras $\mathcal E_{n,m}^q$}

In this paper we study the $C^*$-algebras $\mathcal E_{n,m}^q$ generated by Wick algebras $WE_{n,m}^q$ with the operator of coefficients $T$ given by
\begin{align*}
& T u_1\otimes u_2  =0,\quad T v_1\otimes v_2=0,\quad u_1,u_2\in\mathbb C^n,\ v_1,v_2\in\mathbb C^m,\\
&T u\otimes v  = q v\otimes u,\quad T v\otimes u=\overline q u\otimes v,\quad u\in\mathbb C^n,\ v\in\mathbb C^m,
\end{align*}
for $\mathcal{H}=\mathbb C^n\oplus\mathbb C^m$, $|q|\le 1$.
Note, that $T$ satisfies the braid relation and $||T||=|q|\le 1$ for any $n,m\in\mathbb N$. In particular, the Fock representation $\pi_{F,q}$ exists for $|q|\le 1$ and is faithful on $WE_{n,m}^q$ for $|q|<1$, see the above discussion on general Wick algebras. The $C^*$-algebra $\mathcal E_{n,m}^q$ is generated by isometries $\{s_j\}_{j=1}^n$,  and $\{t_r\}_{r=1}^m$, satisfying commutation relations of the following form
\begin{align}\label{baseqrel}
s_i^*s_j&=0, \quad 1\le i\ne j \le n, \notag
\\
 t_r^*t_s&=0,\quad 0\le r\ne s \le m,\notag
 \\
 s_j^*t_r&=q t_r s_j^*, \quad 0\le j \le n, \ 0\le r\le m. 
\end{align}

They are related to the  $C^*$-algebras of deformed canonical commutation relations $\mathcal G_{\{q_{ij}\}}$, $q_{ij}=\overline{q}_{ji}$, $|q_{ij}|\le 1$, $i,j\in\{1,\ldots,d\}$, $-1<q_{ii}<1$, given by
\begin{equation}\label{qijccr}
a_i^* a_j=\delta_{ij}\mathbf 1 + q_{ij} a_j a_i^*,\quad i,j=\overline{1,d}.
\end{equation}

In one degree of freedom ($d=1$),  $\mathcal G_q$ exists for $q\in [-1,1)$ and
$\mathcal G_q\simeq\mathcal E$ for any $q\in(-1,1)$, see \cite{jsw2}. See also \cite{bied,mac} for more on this algebra. 
The $C^*$-algebra $\mathcal G_{q,d}$ of quon commutation relations with $d$ degrees of freedom was introduced and studied in \cite{BoSpe2,fiv,green,zag} and one has  $\mathcal{G}_{q,d}\simeq \mathcal O_d^0$, for $q<\sqrt{2}-1$. 
The above multiparameter version of quons was considered in \cite{BoSpe,mar,MPe}. For $|q_{ij}|<\sqrt{2}-1$ we get
$\mathcal{G}_{\{q_{ij}\}}\simeq\mathcal O_d^0$.
Further related version have been studied  in \cite{BoLyt,prolett,LiM}.

\subsection{$\mathcal E_{n,m}^q$ in the case $n=m=1$}

In the case  $n=1$, $m=1$, $WE_{1,1}^q$ is generated by isometries $s_1$, $s_2$ subject to the relations \[s_1^*s_2= q s_2 s_1^*.\] 
It is easy to see that its universal $C^*$-algebra $\mathcal E_{1,1}^q$ exists for any $|q|\le 1$.

If $|q|<1$, the main result of \cite{jps2} states that  $\mathcal E_{1,1}^q\simeq\mathcal E_{1,1}^{(0)}=\mathcal O_2^{(0)}$ for any $|q|<1$. In particular the Fock representation of $\mathcal E_{1,1}^q$ is faithful. Notice that the $C^*$\nobreakdash-algebra $\mathcal E_{1,1}^q$ was the only known family of Wick algebras where Conjecture~\ref{q_stab}  holds (in this case $\|T\|<1$ iff $|q|<1$). In particular, even the  isomorphism between $C^*$\nobreakdash-algebra generated by three $q$-commuting isometries and $\mathcal O_3^0$ for all $|q|<1$ is still not established.

The case $|q|=1$ was studied in \cite{kab,prolett,weber}. Here, the additional relation \[s_2s_1=q s_1 s_2\]  holds in $\mathcal E_{1,1}^q$. It was shown that  $\mathcal E_{1,1}^q$ is nuclear for any $|q|=1$. Let $\mathcal M_q$ be the ideal generated by the projections $1-s_1s_1^*$ and $1-s_2s_2^*$. Then $\mathcal E_{1,1}^q/\mathcal M_q\simeq\mathcal A_q$, where $\mathcal A_q$ is the non-commutative torus, see \cite{rieff},
\[
\mathcal A_q=C^*(u_1,u_2\, |\, u_1^*u_1=u_1u_1^*=\mathbf 1,\
u_2^*u_2=u_2u_2^*=\mathbf 1,\ u_2^*u_1=q u_1 u_2^*).
\]
If $q$ is not a root of unity, then the corresponding non-commutative torus $\mathcal A_q$ is simple and $\mathcal M_q$ is the unique largest ideal in $\mathcal E_{1,1}^q$. Let us stress that unlike the case $|q|<1$, the $C^*$-isomorphism class of $\mathcal E_{1,1}^q$ is ``unstable'' with respect to $q$. Namely,
$\mathcal E_{1,1}^{q_1}\simeq \mathcal E_{1,1}^{q_2}$ iff $\mathcal A_{q_1}\simeq\mathcal A_{q_2}$, see \cite{kab,prolett,weber}.

One can consider another higher-dimensional analog of $\mathcal{E}_{1,1}^q$. For a set $\{q_{ij}\}_{i, j=1}^d$ of complex numbers such that $|q_{ij}|\le 1$, $q_{ij}=\overline q_{ji}$, $q_{ii}=1$, and $d>2$,  one can consider a $C^*$\nobreakdash-algebra $\mathcal E_{\{q_{ij}\}}$, generated by $s_j$, $s_j^*$, $j=\overline{1,d}$ subject to the relations
\[
s_j^*s_j=1,\quad s_i^*s_j=q_{ij}s_js_i^*.
\]

The case $|q_{ij}|<1$ was considered in \cite{kuzm_pochek}, where it was proved that $\mathcal{E}_{\{ q_{ij} \}}$ is nuclear and the Fock representation is faithful. It turned out that the fixed point $C^*$-subalgebra of $\mathcal E_{\{q_{ij}\}}$ with respect to the canonical action of $\mathbb T^d$ is an AF-algebra and is independent of $\{q_{ij}\}$. However the conjecture that
$\mathcal E_{\{q_{ij}\}}\simeq\mathcal E_{\{0\}}$ remains open.

The case $|q_{ij}|=1$  was studied in \cite{jeu_pinto,kab,prolett}. It was shown that $\mathcal E_{\{q_{ij}\}}$ is nuclear for any such family $\{q_{ij}\}$ and the Fock representation is faithful.

Let us note, see 
\cite{prolett}, that the $C^*$-algebra $\mathcal E_{\{q_{ij}\}}$ with $|q_{ij}|=1$ is isomorphic to the $C^*$-algebra $\mathcal G_{\{q_ij\}}$ determined by deformed quons. In particular this isomorphism implies that the Fock realization of  $\mathcal G_{\{q_{ij}\}}$ is faithful, so the Fock representation can be considered as the universal representation of $\mathcal G_{\{q_ij \}}$. We stress also that apart the case $|q_{ij}|<1$ the $C^*$-isomorphism class of $\mathcal E_{\{q_{ij}\}}$ with $|q_{ij}|=1$ depends on the $C^*$-isomorphism class of the non-commutative torus $\mathbf T_{\{qij\}}$,
\[
\mathbf T_{\{q_{ij}\}}=C^* \left( u_i,\ u_i^*=u_i^{-1},\ u_ju_i=q_{ij}u_j u_i\right),
\]
and is unstable at any point.

\subsection{$\mathcal E_{n,m}^q$ in the case $n,m\geq 2$} 

In our article, we focus on the study of $\mathcal E_{n,m}^q$ with $n,m\ge 2$ (see \cite{yakym} for the case $n=1$, $m\ge 2$) and we also consider a multiparameter case. In  the one parameter case, the analysis is separated into two conceptually different cases, $|q|<1$ and $|q|=1$.

If $|q|<1$, we show that $\mathcal E_{n,m}^q\simeq\mathcal E_{n,m}^0= \mathcal O_{n+m}^{(0)}$, where the latter is the Cuntz-Toeplitz algebra with $n+m$ generators.

For the case $|q|=1$, we stress out that $\mathcal E_{n,m}^q$ is isomorphic to the Rieffel deformation of $\mathcal O_n^{(0)}\otimes\mathcal O_m^{(0)}$ implying in particular the nuclearity of $\mathcal E_{n,m}^q$. We show that it  contains a unique largest ideal $\mathcal M_q$, and we consider the quotient $\mathcal O_n\otimes_q\mathcal O_m:=\mathcal E_{n,m}^q/\mathcal M_q$, in fact, even for a multiparameter  $\Theta=(q_{ij})$ with $\lvert q_{ij}\rvert=1$, denoting the object $\mathcal O_n\otimes_\Theta\mathcal O_m$. We show that $\mathcal O_n\otimes_q\mathcal O_m$ and $\mathcal O_n\otimes_\Theta\mathcal O_m$ are  simple and purely infinite. We use  Kirchberg-Philips's classification Theorem, see \cite{Kirchberg,Philips}, to get one of our main results, namely
\[
\mathcal O_n\otimes_q\mathcal O_m\simeq\mathcal O_n\otimes\mathcal O_m\qquad\textnormal{and}\qquad \mathcal O_n\otimes_\Theta\mathcal O_m\simeq\mathcal O_n\otimes\mathcal O_m
\]
for any $q, q_{ij}\in\mathbb C$, $\lvert q\rvert=\lvert q_{ij}\rvert = 1$.  Next we show that the isomorphism class of $\mathcal M_q$ is independent of $q$ and consider $\mathcal E_{n,m}^q$ as an (essential) extension of $\mathcal O_n\otimes\mathcal O_m$ by $\mathcal M_q$ and study the corresponding $\mathsf{Ext}$ group. In particular, if $\gcd (n-1,m-1)=1$, this group is zero. Thus in this case, $\mathcal{E}_{n,m}^q$ and $\mathcal{E}_{n,m}^1$ both determine the zero class in $\mathsf{Ext}(\mathcal{O}_n \otimes_q \mathcal{O}_m, \mathcal{M}_q)$.

We stress that unlike the case of extensions by compacts, one cannot immediately deduce that two trivial essential extensions are isomorphic. So the problem of an isomorphism $\mathcal{E}_{n,m}^q \simeq \mathcal{E}_{n,m}^1$ remains open.

\subsection{Relation with deformed CCR}

Finally we present the relation of $\mathcal{E}_{n,m}$, $n,m\ge 2$ with multi-component commutation relations, see \cite{BoLyt, Dalet, LiM}.

Take $k\in (0,1)$ and $q\in\mathbb C$, $|q|=1$. Construct $\mathcal H=\mathbb C^n\oplus\mathbb C^m$, $n$, $m\ge 2$ and define $T\colon\mathcal H^{\otimes 2}\rightarrow\mathcal H^{\otimes 2}$ as follows
\begin{align*}
T u_1\otimes u_2 &= k\, u_2\otimes u_1,\quad \mbox{if either}\ u_1,u_2\in\mathbb C^n\ \mbox{or}\ u_1,u_2\in\mathbb C^m\\
T u\otimes v &= q\, v\otimes u,\quad \mbox{if}\ u\in\mathbb C^n,\ v\in\mathbb C^m.
\end{align*}
Denote the corresponding Wick algebra by $WE_{n,m}^{q,k}$ and its universal $C^*$-algebra by $\mathcal E_{n,m}^{q,k}$. This $C^*$-algebra is generated by $s_j$, $t_r$, $j=\overline{1,n}$, $r=\overline{1,m}$, subject to the relations
\begin{align}\label{d_plectons}
s_i^*s_j&=\delta_{ij}\left(\mathbf 1 + k\, s_j s_i^*\right),\nonumber\\
t_r^* t_l&=\delta_{rl}\left(\mathbf 1+ k\, t_l t_r^*\right),\nonumber\\
s_j^* t_r &= q\, t_r s_j^*.
\end{align}
Notice that in  $\mathcal E_{n,m}^{q,k}$ with $|q|=1$, $k\in (0,1)$, the relations 
\begin{equation}\label{d_plectons2}
 t_r s_j= q\, s_j t_r, \quad j=\overline{1,n}, \ r=\overline{1,m},
\end{equation}
hold as well. Indeed, for $B_{jr} = t_r s_j- q\, s_j t_r$ we have 
$B_{jr}^*B_{jr} 
=k^2 B_{jr}B_{jr}^*
$
and $B_{jr}=0$.

Relations \eqref{d_plectons}, \eqref{d_plectons2} can be regarded as an example of system considered in \cite{BoLyt, Dalet} in the case of finite count of degrees of freedom. 

One can show, \cite{prolett},  that $\mathcal E_{n,m}^{q,k}\simeq\mathcal E_{n,m}^q$ for $k\in(0,1)$. Hence, in particular, one of our main results says that the $C^*$-algebra generated by (\ref{d_plectons}) is an extension of $\mathcal O_n\otimes\mathcal O_m$. One more important corollary of the isomorphism is that the Fock representation of relations (\ref{d_plectons}), (\ref{d_plectons2}) is faithful.

Notice that for $k=\pm 1$ we get a discrete analogue of commutation relations for generalized statistics introduced in \cite{LiM}.

\section{The case $|q|<1$}
We start with some lemmas. Let $\Lambda_n$ denote the set of all words in alphabet $\{\overline{1,n}\}$. For any non-empty $\mu=(\mu_1,\ldots,\mu_k)$, and a family of elements $b_1$, \dots, $b_n$, we denote by $b_{\mu}$ the product $b_{\mu_1}\cdots b_{\mu_k}$; we also put $b_{\emptyset}=\mathbf{1}$. In this section we assume that any word $\mu$ belongs to $\Lambda_n$.
\begin{lemma}
Let $Q=\sum_{i=1}^n s_i s_i^*$, then
\[
\sum_{|\mu|=k} s_{\mu}Q s_{\mu}^* =\sum_{|\nu|=k+1} s_{\nu}s_{\nu}^*.
\]
\end{lemma}
\begin{proof}
Straightforward.
\end{proof}

\begin{lemma}
For any $x\in\mathcal{E}_{n,m}^q$ one has
\[
\Bigl\|\sum_{|\mu|=k}s_{\mu}x s_{\mu}^*\Bigr\| \le \|x\|.
\]
\end{lemma}
\begin{proof}
 $1$. First prove the claim for positive $x$. In this case one has $0\le x\le \|x\|\mathbf{1}$. Hence $0\le s_{\mu} x s_{\mu}^* \le \|x\| s_{\mu}s_{\mu}^*$, and
\[
\Bigl\| \sum_{|\mu|=k}s_{\mu}x s_{\mu}^*\Bigr\|\le \|x\|\cdot \Bigl\|\sum_{|\mu|=k} s_{\mu}s_{\mu}^*\Bigr\|.
\]
Note that $s_{\mu}^* s_{\lambda} =\delta_{\mu\,\lambda}$, $\mu,\lambda\in\Lambda_n$, $|\mu|=|\lambda|=k$, implying that $\{s_{\mu}s_{\mu}^*\ |\ |\mu|=k\}$ form a family of pairwise orthogonal projections. Hence $\|\sum_{|\mu|=k} s_{\mu}s_{\mu}^*\|=1$, and the statement for positive $x$ is proved.

 $2$. For any $x\in\mathcal{E}_{n,m}^q$. write $A=\sum_{|\mu|=k}s_{\mu} x s_{\mu}^*$, then $A^*=\sum_{|\mu|=k}s_{\mu} x^* s_{\mu}^*$ and
\[
A^* A=\sum_{|\mu|=k}s_{\mu} x^* x s_{\mu}^*.
\]
Then by the proved above,
\[
\|A\|^2 =\|A^*A\|\le \|x^*x\|=\|x\|^2. \qedhere
\]
\end{proof}

Construct $\widetilde{t}_l=(\mathbf{1}-Q) t_l$,  $l=\overline{1,m}$.

\begin{lemma}
The following commutation relations hold
\begin{align*}
&s_i^*\widetilde{t}_l=0,\quad i=\overline{1,n},\quad l=\overline{1,m},
\\
&\widetilde{t}_r^*\widetilde{t}_l=0,\quad l\ne r\quad l,r=\overline{1,m},
\\
& \widetilde{t}_r^*\widetilde{t}_r=\mathbf{1}-|q|^2 Q>0,\quad r=\overline{1,m}.
\end{align*}
\end{lemma}

\begin{proof} We have $s_i^* (\mathbf{1}-Q)=0$, implying that $s_i^*\widetilde{t}_l=0$ for any $i=\overline{1,n}$, and $l=\overline{1,m}$.

Further,
\begin{align*}
\widetilde{t}_r^*\widetilde{t}_l&=t_r^* (\mathbf{1}-Q)t_l=t_r^* t_l-
\sum_{i=1}^n t_r^* s_i s_i^* t_l=\delta_{rl}\mathbf{1}-\sum_{i=1}^n |q|^2 s_i t_r^* t_l s_i^*=\\
&=\delta_{rl}(\mathbf{1}-|q|^2 Q).\qedhere
\end{align*}
\end{proof}

\begin{proposition}\label{gener1}
For any $r=\overline{1,m}$, one has
\[
t_r=\sum_{k=0}^{\infty}\sum_{|\mu|=k} q^k s_{\mu}\widetilde{t}_r s_{\mu}^*.
\]
In particular, the family $\{s_i,\widetilde{t}_r,\ i=\overline{1,n},\ r=\overline{1,m}\}$ generates $\mathcal{E}_{n,m}^q$.
\end{proposition}

\begin{proof}
Put $M_k^{r}=\sum_{|\mu|=k} q^k s_{\mu}\widetilde{t}_r s_{\mu}^*$, $k\in\mathbb{Z}_{+}$. Then
\[
M_0^r=t_r- Q t_r=t_r-\sum_{|\mu|=1} s_{\mu}s_{\mu}^* t_r,
\]
and
\begin{align*}
M_k^r&=\sum_{|\mu|=k} q^k s_{\mu} (\mathbf{1}-Q) t_r s_{\mu}^*=
\sum_{|\mu|=k}  s_{\mu} (\mathbf{1}-Q) s_{\mu}^*t_r =\\
&=\sum_{|\mu|=k}  s_{\mu} s_{\mu}^* t_r - \sum_{|\mu|=k+1} s_{\mu} s_{\mu}^* t_r.
\end{align*}
Then
\[
S_N^r=\sum_{k=0}^{N} M_k^r=t_r-\sum_{|\mu|=N+1}s_{\mu}s_{\mu}^* t_r=
t_r- q^{N+1}\sum_{|\mu|=N+1}s_{\mu}t_rs_{\mu}^*.
\]
Since $\|\sum_{|\mu|=N+1}s_{\mu}t_rs_{\mu}^*\|\le \|t_r\|=1$ one has that $S_N^r\rightarrow t_r$ in $\mathcal{E}_{n,m}^q$ as $N\to\infty$.
\end{proof}

Suppose that $\mathcal{E}_{n,m}^q$ is realized by Hilbert space operators. Consider the left polar decomposition $\widetilde{t}_r =\widehat{t}_r\cdot c_r$, where $c_r^2=\widetilde{t}_r^*\widetilde{t}_r=\mathbf{1}-|q|^2 Q>0$, implying that $\widehat{t}_r$ is an isometry and
\[
\widehat{t}_r=\widetilde{t}_r c_r^{-1}\in\mathcal{E}_{n,m}^q,\quad r=\overline{1,m}.
\]

\begin{lemma}
The following commutation relations hold
\begin{align*}
s_i^*\widehat{t}_r&=0,\quad i=\overline{1,n},\ r=\overline{1,m},
\\ \widehat{t}_r^*\widehat{t}_l&=\delta_{rl}\mathbf{1},\quad r,l=\overline{1,m}.
\end{align*}
\end{lemma}

\begin{proof}
Indeed, for any $i=\overline{1,n}$, and $r=\overline{1,m}$. one has
\[
s_i^* \widehat{t}_r=s_i^*\widetilde{t}_r\, c_r^{-1}=0,
\]
and
\[
\widehat{t}_r^*\widehat{t}_l=c_r^{-1}\widetilde{t}_r^*\widetilde{t}_l c_r^{-1}=0,\quad r\ne l.
\qedhere
\]
\end{proof}

Summing up the results stated above, we get the following

\begin{theorem}\label{zer_q_gen}
Let  $\widehat{t}_r=(\mathbf{1}-Q)t_r (\mathbf{1}-|q|^2 Q)^{-\frac{1}{2}}$, $r=\overline{1,m}$. Then the family $\{s_i,\widehat{t}_r\}_{i=1}^{n}{}_{r={1}}^{m}$ generates $\mathcal{E}_{n,m}^q$, and
\[
s_i^*s_j=\delta_{ij}\mathbf{1},\quad \widehat{t}_r^* \widehat{t}_l=\delta_{rl}\mathbf{1},\quad s_i^* \widehat{t}_r =0,\qquad i,j=\overline{1,n},\ r,l=\overline{1,m}.
\]
\end{theorem}

\begin{proof}
It remains to note that $\widetilde{t}_r=\widehat{t}_r (1-|q|^2 Q)^{\frac{1}{2}}$, so $\widetilde{t}_r\in C^* (\widehat{t}_r,\, Q)$, so by Proposition \ref{gener1} the elements $s_i$, $\widehat{t}_r$, $i=\overline{1,n}$, $r=\overline{1,m}$, generate $\mathcal{E}_{n,m}^q$.
\end{proof}

\begin{corollary}
Denote by $v_i$, $i=\overline{1,n}+m$, the isometries generating $\mathcal{E}_{n,m}^0=\mathcal{O}_{n+m}^{(0)}$. Then Theorem \ref{zer_q_gen} implies that the correspondence
\[
v_i\mapsto s_i,\ i=\overline{1,n},\quad v_{n+r}\mapsto \widehat{t}_r,\ r=\overline{1,m},
\]
extends uniquely to a surjective homomorphism $\varphi\colon\mathcal{E}_{n,m}^0\rightarrow\mathcal{E}_{n,m}^q$.
\end{corollary}

Our next aim is to construct the inverse homomorphism $\psi\colon\mathcal{E}_{n,m}^q\rightarrow\mathcal{E}_{n,m}^0$. To do it, put
\[
\widetilde{Q}=\sum_{i=1}^n v_i v_i^*\quad \widetilde{w}_r=v_{n+r}(1-|q|^2\widetilde{Q})^{\frac{1}{2}}, \qquad r=\overline{1,m}.
\]
 Then $\widetilde{w}_r^*\widetilde{w}_r=1-|q|^2\widetilde{Q}$, and $\widetilde{w}_r^*\widetilde{w}_l=0$ if $r\ne l$, $r,l=\overline{1,m}$. Construct
\[
w_r=\sum_{k=0}^{\infty}\sum_{|\mu|=k}q^k v_{\mu}\widetilde{w}_r v_{\mu}^*,\quad r=\overline{1,m},
\]
where $\mu$ runs over $\Lambda_n$, and set as above $v_{\mu}=v_{\mu_1}\cdots v_{\mu_k}$. Note that the series above converges with respect to norm in $\mathcal{E}_{n,m}^0$.

\begin{lemma}
The following commutation relations hold
\[
w_r^*w_l=\delta_{rl}\mathbf{1},\quad v_i^* w_r=q w_r v_i^*,\quad i=\overline{1,n},\ r,l=\overline{1,m}.
\]
\end{lemma}

\begin{proof}
First we note that $v_i^*\widetilde{w}_r=0$, and $\widetilde{w}_r^* v_i=0$ for any $i=\overline{1,n}$, and $j=\overline{1,m}$, implying that
\[
v_{\delta}^*\widetilde{w}_r=0,\quad \widetilde{w}_r^* v_{\delta}=0,\quad \mbox{for any nonempty}\ \delta\in\Lambda_n,\ r=\overline{1,m}.
\]
Let $|\lambda|\ne |\mu|$, $\lambda,\mu\in\Lambda_n$. If $|\lambda|>|\mu|$, then $\lambda=\widehat{\lambda}\gamma$ with $|\lambda|=|\mu|$ and
$
v_{\lambda}^*v_{\mu}=\delta_{\widehat{\lambda}\mu}v_{\gamma}^*.
$
Otherwise $\mu=\widehat{\mu}\beta$, $|\widehat{\mu}|=|\lambda|$ and
$
v_{\lambda}^*v_{\mu}=\delta_{\lambda\widehat{\mu}}v_{\beta}.
$
So, if $|\lambda|>|\mu|$ one has
\[
v_{\lambda}\widetilde{w}_r^* v_{\lambda}^* v_{\mu}\widetilde{w}_r v_{\mu}^*=
\delta_{\widehat{\lambda}\mu} v_{\lambda}\widetilde{w}_r^* v_{\gamma}^*\widetilde{w}_r v_{\mu}=0,
\]
and if $|\mu|>|\lambda|$, then
\[
v_{\lambda}\widetilde{w}_r^* v_{\lambda}^* v_{\mu}\widetilde{w}_r v_{\mu}^*=
\delta_{\lambda\widehat{\mu}} v_{\lambda}\widetilde{w}_r^* v_{\beta}\widetilde{w}_r v_{\mu}=0.
\]
Since $v_{\mu}^*v_{\lambda}=\delta_{\mu\lambda}\mathbf{1}$, if $|\mu|=|\lambda|$, one has
\begin{align*}
 w_r^*w_r&=\lim_{N\rightarrow\infty}\Bigl(\sum_{k=0}^N \sum_{|\lambda|=k}|q|^{k} v_{\lambda}\widetilde{w}_r^* v_{\lambda}^*\Bigr)\cdot
\Bigl(\sum_{l=0}^N \sum_{|\mu|=l}|q|^{l} v_{\mu}\widetilde{w}_r v_{\mu}^*\Bigr)
\\
&=\lim_{N\rightarrow\infty}\sum_{k,l=0}^N \sum_{|\lambda|=k,|\mu|=l}|q|^{k+l} v_{\lambda}\widetilde{w}_r^* v_{\lambda}^* v_{\mu}\widetilde{w}_r v_{\mu}^*
\\
&=
\lim_{N\rightarrow\infty}\sum_{k=0}^N \sum_{|\lambda|,|\mu|=k,}|q|^{2k} v_{\lambda}\widetilde{w}_r^* v_{\lambda}^* v_{\mu}\widetilde{w}_r v_{\mu}^*
\\
&=\lim_{N\rightarrow\infty}\sum_{k=0}^N \sum_{|\mu|=k}|q|^{2k} v_{\mu}\widetilde{w}_r^*\widetilde{w}_r v_{\mu}^*=
\lim_{N\rightarrow\infty}\sum_{k=0}^N \sum_{|\mu|=k}|q|^{2k} v_{\mu}(\mathbf{1}-|q|^2\widetilde{Q}^2) v_{\mu}^*
\\
&=\lim_{N\rightarrow\infty}\sum_{k=0}^N \Bigl (\sum_{|\mu|=k}|q|^{2k} v_{\mu}v_{\mu}^*-\sum_{|\mu|=k+1} |q|^{2k+2}v_{\mu} v_{\mu}^* \Bigr)
\\
&=\lim_{N\rightarrow\infty}\Bigl(\mathbf{1}-|q|^{2N+2}\sum_{|\mu|=N+1} v_{\mu}v_{\mu}^*\Bigr)=\mathbf{1}.
\end{align*}
Since $\widetilde{w}_r^*\widetilde{w}_l=0$, $r\ne l$, the same arguments as above imply that $w_r^* w_l=0$, $r\ne l$.

For any non-empty $\mu\in\Lambda_n$ write $\sigma(\mu)=\emptyset$ if $|\mu|=1$, and
$\sigma(\mu)=(\mu_2,\ldots,\mu_k)$ if ${|\mu|=k>1}$. Further, for any $i=\overline{1,n}$, $r=\overline{1,m}$ one has
\begin{align*}
v_i^* w_r&=\sum_{k=0}^{\infty}\sum_{|\mu|=k} q^k s_i^* v_{\mu}\widetilde{w}_r v_{\mu}^*=
v_i^*\widetilde{w}_r+\sum_{k=1}^{\infty}\sum_{|\mu|=k} q^k \delta_{i\mu_1} v_{\sigma(\mu)}\widetilde{w}_r v_{\sigma(\mu)}^*v_i^*
\\
&=q\sum_{k=0}^{\infty}\sum_{|\mu|=k} q^k v_{\mu}\widetilde{w}_r v_{\mu}^*v_i^*=q w_r v_i^*.
\qedhere
\end{align*}
\end{proof}

\begin{lemma}
For any $r=\overline{1,m}$, one has $\widetilde{w}_r=(\mathbf{1}-\widetilde{Q})w_r$.
\end{lemma}

\begin{proof}
First note that $(\mathbf{1}-\widetilde{Q})v_i=0$, $i=\overline{1,n}$, implies that
\[
(\mathbf{1}-\widetilde{Q})v_{\mu}=0,\quad |\mu|\in\Lambda_n,\ \mu\ne\emptyset.
\]
Then
\begin{align*}
(\mathbf{1}-\widetilde{Q})w_r&=(\mathbf{1}-\widetilde{Q})\Bigl(\sum_{k=0}\sum_{|\mu|=k} q^k v_{\mu}\widetilde{w}_r v_{\mu}^*\Bigr)
\\
&=(\mathbf{1}-\widetilde{Q})\widetilde{w}_r+\sum_{k=1}\sum_{|\mu|=k} q^k (\mathbf{1}-\widetilde{Q})v_{\mu}\widetilde{w}_r v_{\mu}^*=(\mathbf{1}-\widetilde{Q})\widetilde{w}_r.
\end{align*}

To complete the proof it remains to note that $\widetilde{Q}v_{n+r}=0$, $r=\overline{1,m}$. So,
\[
\widetilde{Q}\widetilde{w}_r=\widetilde{Q}v_{n+r}(\mathbf{1}-|q|^2\widetilde{Q})^{\frac{1}{2}}=0.
\]
\end{proof}

\begin{theorem}\label{q_ser_gen}
Let $v_i$, $i=\overline{1,n}+m$, be the isometries generating $\mathcal{E}_{n,m}^0$, and $\widetilde{Q}=\sum_{i=1}^n v_i v_i^*$. Put
\[
\widetilde{w}_r=v_{n+r}(\mathbf{1}-|q|^2\widetilde{Q})^{\frac{1}{2}}\quad \mbox{and}\quad w_r=\sum_{k=0}\sum_{|\mu|=k} q^k v_{\mu}\widetilde{w}_r v_{\mu}^*.
\]
Then
\[
v_i^*v_j=\delta_{ij}\mathbf{1},\quad w_r^* w_l=\delta_{rl}\mathbf{1},\quad v_i^* w_r=q w_r v_i^*,\quad i,j=\overline{1,n},\ r,l=\overline{1,m}.
\]
Moreover, the family $\{v_i,w_r\}_{i=1}^{n}{}_{r=1}^{m}$ generates $\mathcal{E}_{n,m}^0$.
\end{theorem}

\begin{proof}
We need to prove only the last statement of the theorem.  We have
\[
v_{n+r}=\widetilde{w}_r (\mathbf{1}-|q|^2\widetilde{Q})^{-\frac{1}{2}}=(\mathbf{1}-\widetilde{Q})
w_r(\mathbf{1}-|q|^2\widetilde{Q})^{-\frac{1}{2}}\in C^*(w_r, v_i,\ i=\overline{1,n}).
\]
Hence $v_i$, $w_r$, $i=\overline{1,n}$, $r=\overline{1,m}$, generate $\mathcal{E}_{n,m}^0$.
 \end{proof}

\begin{corollary}\label{cor2}
The statement of Theorem \ref{q_ser_gen} and the universal property of $\mathcal{E}_{n,m}^q$ imply the existence of a surjective homomorphism $\psi\colon\mathcal{E}_{n,m}^q\rightarrow\mathcal{E}_{n,m}^0$ defined by
\[
\psi(s_i)=v_i,\quad \psi(t_r)=w_r,\quad i=\overline{1,n},\ r=\overline{1,m}.
\]
\end{corollary}

Now we are ready to formulate the main result of this section.

\begin{theorem}\label{qstab_thm}
For any $q\in\mathbb{C}$, $|q|<1$, one has an isomorphism $\mathcal{E}_{n,m}^q\simeq\mathcal{E}_{n,m}^0$.
\end{theorem}

\begin{proof}
In Theorem \ref{zer_q_gen}, we constructed the surjective homomorphism $\varphi\colon\mathcal{E}_{n,m}^0\rightarrow\mathcal{E}_{n,m}^q$ defined by
\[
\varphi(v_i)=s_i,\quad \varphi(v_{n+r})=\widehat{t}_r,\quad i=\overline{1,n},\ r=\overline{1,m}.
\]
Show that $\psi\colon\mathcal{E}_{n,m}^q\rightarrow\mathcal{E}_{n,m}^0$ from Corollary \ref{cor2} is the inverse of $\varphi$. Indeed, the equalities  ${\psi(s_i)=v_i}$, $i=\overline{1,n}$, imply that
\[
\psi (\mathbf{1}-Q)=\mathbf{1}-\widetilde{Q}.
\]
Then, since $\psi(t_r)=w_r$, we get
\[
\psi (\widetilde{t}_r)=\psi ((\mathbf{1}-Q)t_r)=(\mathbf{1}-\widetilde{Q})w_r=\widetilde{w}_r,\quad r=\overline{1,m},
\]
and
\[
\psi(\widehat{t}_r)=\psi(\widetilde{t}_r (\mathbf{1}-|q|^2 Q)^{-\frac{1}{2}})=
\widetilde{w}_r (\mathbf{1}-|q|^2 \widetilde{Q})^{-\frac{1}{2}}=v_{n+r},\quad r=\overline{1,m}.
\]
So, $\psi\varphi(v_i)=\psi (s_i)=v_i$, $\psi\varphi(v_{n+r})=\psi(\widehat{t}_r)=v_{n+r}$, $i=\overline{1,n}$, $r=\overline{1,m}$, and
\[
\psi\varphi= id_{\mathcal{E}_{n,m}^0}.
\]

Show that $\varphi\psi=id_{\mathcal{E}_{n,m}^q}$. Indeed,
\[
\varphi(\widetilde{w}_r)=\varphi(v_{n+r}(\mathbf{1}-|q|^2\widetilde{Q})^{\frac{1}{2}})=
\widehat{t}_r (\mathbf{1}-|q|^2 Q)^{\frac{1}{2}}=\widetilde{t}_r,\quad r=\overline{1,d}.
\]
Then for any $r=\overline{1,m}$, one has
\[
\varphi(w_r)=\sum_{k=0}\sum_{|\mu|=k}q^k\varphi(v_{\mu})\varphi(\widetilde{w}_r)\varphi{v_{\mu}^*}=
\sum_{k=0}\sum_{|\mu|=k} q^ks_{\mu} \widetilde{t}_r s_{\mu}^*= t_r.
\]
So, $\varphi\psi(s_i)=\varphi(v_i)=s_i$, $\varphi\psi(t_r)=\varphi(w_r)=t_r$, $i=\overline{1,n}$, $r=\overline{1,m}$.
 \end{proof}

\section {The case $|q|=1$}
In this section, we discuss the case $|q|=1$. Notice that for $|q|=1$, formula \eqref{baseqrel} implies  $t_js_i = q s_it_j$, $i=\overline{1,n}$, $j=\overline{1, m}$. Indeed, one has just put $B_{ij}=t_j s_i - q s_i t_j$ and 
check that $B_{ij}^* B_{ij}=0$, $i\ne j$.

\subsection{Auxiliary results}
In this subsection we collect some general facts about $C^*$-dynamical systems, crossed products and Rieffel deformations which we will use in our considerations.

\subsubsection{Fixed point subalgebras}\label{group_action}
First we recall how properties of a fixed point subalgebra of a $C^*$-algebra with an action of a compact group are related to properties of the whole algebra.

\begin{definition}
Let $\mathcal A$ be a $C^*$-algebra with an action $\gamma$ of a compact group $G$. A~fixed point subalgebra ${\mathcal A}^\gamma$ is a subset of all $a \in\mathcal A$ such that $\gamma_g(a) = a$ for all $g \in G$.
\end{definition}

Notice  that for every action of a compact group $G$ on a $C^*$-algebra $\mathcal A$ one can construct a faithful conditional expectation 
$E_{\gamma} :\mathcal A \rightarrow\mathcal{ A}^\gamma$ onto the fixed point subalgebra, given by
\[ E_{\gamma}(a) = \int_{G} \gamma_g(a) d \lambda, \]
where $\lambda$ is the Haar measure on $G$.

 A homomorphism $\varphi :\mathcal A \rightarrow\mathcal B$ between $C^*$-algebras with actions $\alpha$ and $\beta$ of a compact group $G$ is called equivariant if
 \[ \varphi \circ \alpha_g = \beta_g \circ \varphi \text{ for any } g \in G.
 \]

\subsubsection{Crossed products}

Given a locally compact group $G$ and a $C^*$-algebra $\mathcal A$ with a $G$-action $\alpha$, consider the full crossed product $C^*$-algebra $\mathcal{A} \rtimes_\alpha G$, see \cite{williams}. One has two natural embeddings into the multiplier algebra  $M(\mathcal{A} \rtimes_\alpha G)$,
\begin{gather*}
i_{\mathcal A} : \mathcal A \rightarrow M(\mathcal{A} \rtimes_\alpha G), \ i_G : G \rightarrow M(\mathcal{A} \rtimes_\alpha G), 
\\
(i_{\mathcal A}(a) f)(s) = a f(s), \quad (i_G(t)f)(s) = \alpha_t(f(t^{-1} s)), \quad t,s\in G, \ a \in\mathcal  A,
\end{gather*}
for $f \in C_c(G, \mathcal A)$.
\begin{remark}\label{rem2}
Obviously, $i_G(s)$ is a unitary element of $M(\mathcal A\rtimes_{\alpha} G)$ for any $s\in G$. Recall that $i_G$ determines the following homomorphism denoted also by $i_G$
\[
i_G\colon C^*(G)\rightarrow M(\mathcal{A}\rtimes_{\alpha} G)
\]
defined by
\[
i_G (f)=\int_G f(s) i_G(s) d\lambda(s),
\]
where $\lambda$ is the left Haar measure on $G$.

Notice that for any $g\in C_c(G,\mathcal A)$ one has
\[
(i_G(f) g)(t)=f\cdot_{\alpha} g,
\]
where $\cdot_\alpha$ denotes the product in $\mathcal{A}\rtimes_{\alpha} G$. In particular, when $\mathcal A$ is unital we can identify $i_G(f)$ with $f\cdot_\alpha \mathbf{1}_{\mathcal A}$, and in fact $i_G$ maps $C^*(G)$ into $\mathcal{A} \rtimes_\alpha G$. Also notice that
\[
i_G(t) i_{\mathcal A} (a) i_G(t)^{-1}  =i_{\mathcal A}( \alpha _t(a)) \in M(\mathcal{A} \rtimes_\alpha G).
\]
\end{remark}

If  $\varphi$ is an equivariant homomorphism between $C^*$-algebras $\mathcal A$ with a $G$-action $\alpha$ and $\mathcal B$ with a $G$-action $\beta$, then one can define the homomorphism
\[
\varphi \rtimes G : \mathcal{A} \rtimes_\alpha G \rightarrow\mathcal{B} \rtimes_\beta G, \ (\varphi \rtimes G)(f)(t) = \varphi(f(t)),
\quad f\in C_c(G,\mathcal A).
\]
Let $\mathcal A$ be a unital $C^*$-algebra with $G$-action $\alpha$. Then $\iota_{\mathcal A}\colon\mathbb{C}\rightarrow\mathcal A$,
\[
\iota_{\mathcal A}(\lambda)=\lambda\mathbf 1_{\mathcal A},
\]
is an equivariant homomorphism, where $G$ acts trivially on $\mathbb{C}$. Since $\mathbb{C}\rtimes G=C^* (G)$, one has that
\[
\iota_{\mathcal A}\rtimes G\colon C^*(G)\rightarrow \mathcal{A}\rtimes_{\alpha} G.
\]
In fact, in this case we have
\begin{equation}\label{unital_embedding}
    \iota_{\mathcal A} \rtimes G = i_G,
\end{equation}
where $i_G\colon C^*(G)\rightarrow\mathcal{A}\rtimes_{\alpha} G$ is described in Remark \ref{rem2}. Indeed, for any $g\in G_c (G,\mathcal A)$ one has
\begin{equation*}
    \begin{split}
        (i_G(f)\cdot_\alpha g)(s) & = \int_G f(t) \alpha_t(g(t^{-1}s)) dt
        \\
        & = \int_G f(t) 1_A \alpha_t(g(t^{-1}s)) dt
        \\
        & = ((f(\cdot)1_{\mathcal A}) \cdot_\alpha g)(s)=((\iota_{\mathcal A} \rtimes G)(f)\cdot_{\alpha}g)(s),
    \end{split}
\end{equation*}
implying $i_G(f)=(\iota_{\mathcal A}\rtimes G)(f)$ for any $f\in C^*(G)$.

\subsubsection{Rieffel's deformation}

Below, we recall some basic facts on Rieffel's deformations. Given a $C^*$-algebra $\mathcal A$ equipped with an action $\alpha$ of $\mathbb{R}^n$ and a skew symmetric matrix
$\Theta \in M_n(\mathbb R)$, one can construct the Rieffel deformation of $\mathcal A$, denoted by $\mathcal A_\Theta$, see \cite{anderson,rieffel_primar}. In particular the elements $a \in \mathcal  A$ such that $x \mapsto \alpha_x(a) \in C^\infty(\mathbb{R}^n,\mathcal  A)$ form a dense subset $\mathcal A_\infty$ in $\mathcal A_\Theta$ and for any $a,b\in \mathcal A_\infty$ their product in $\mathcal A_\Theta$ is given by the following oscillatory integral (see \cite{rieffel_primar}):
\begin{equation}\label{rieffel_product}
    a \cdot_\Theta b :=
\int_{\mathbb{R}^n} \int_{\mathbb{R}^n} \alpha_{\Theta(x)}(a) \alpha_y(b) e^{2 \pi i \langle x, y \rangle} dx dy,
\end{equation}
where $\langle \cdot,\cdot\rangle$ is a scalar product in $\mathbb R^n$. The mapping $\alpha^\Theta_x\colon a\to \alpha_{x}(a)$, $a\in \mathcal A_\infty$, extends naturally to an action $\alpha^\Theta$ of $\mathbb R^n$ on $\mathcal A_\Theta$.

Given an equivariant $*$-homomorphism $\phi$ between $C^*$-algebras $\mathcal A$ and $\mathcal B$ with actions of $\mathbb{R}^n$, one can define a $*$-homomorphism $\phi_\Theta : \mathcal A_\Theta \rightarrow \mathcal B_\Theta$, which is also equivariant with respect to the induced actions. Moreover, $\phi_\Theta$ is injective iff $\phi$ is injective, see e.g. \cite{Kasprzak_rieffel}.

The next result follows directly from Lemma 3.5 of \cite{Kasprzak_rieffel}. 

\begin{proposition}\label{double_deformation}
The mapping $\mathrm{id} : \mathcal A \rightarrow (\mathcal A_\Theta)_{-\Theta}$ is an equivariant $*$-isomorphism.
\end{proposition}

In what follows, we will be interested in periodic actions of $\mathbb{R}^n$, i.e., we assume that $\alpha$ is an action of $\mathbb{T}^n$.
Given a character $\chi \in \widehat{\mathbb{T}}^n \simeq \mathbb{Z}^n$, consider
\[
\mathcal A_\chi = \{ a \in \mathcal A : \alpha_z(a) = \chi(z)a \text{ for every } z \in \mathbb{T}^n \} .
\]
Then
\[
\mathcal A=\overline{\bigoplus_{\chi\in\mathbb Z^n} \mathcal A_{\chi}},
\]
where some terms could be equal to zero, and $\mathcal A_{\chi_1}\cdot \mathcal A_{\chi_2}\subset \mathcal A_{\chi_1+\chi_2}$, $\mathcal A_{\chi}^*=\mathcal A_{-\chi}$. So,  $\mathcal A_\chi$, $\chi \in \mathbb Z^n$, can be treated as homogeneous components of
$\mathbb Z^n$-grading on $\mathcal A$.

For a periodic action $\alpha$ of $\mathbb{R}^n$ on a $C^*$-algebra $\mathcal A$ and a skew-symmetric matrix $\Theta \in M_n(\mathbb{R})$, construct the Rieffel deformation $\mathcal A_\Theta$. Notice that all
homogeneous elements belong to $\mathcal A_{\infty}$. Apply formula (\ref{rieffel_product}) to $a \in \mathcal A_p$, $b \in \mathcal A_q$:
\begin{equation*}
    \begin{split}
        a \cdot_\Theta b & =  \int_{\mathbb{R}^n}
        \int_{\mathbb{R}^n} e^{2 \pi i \langle \Theta(x), p \rangle} a e^{2 \pi i \langle y, q \rangle}
        b e^{2 \pi i \langle x, y \rangle} dx\, dy
        \\
        & = a \cdot b \int_{\mathbb{R}^n} e^{2 \pi i \langle y, q \rangle}
        \int_{\mathbb{R}^n} e^{2 \pi i \langle x, -\Theta(p) \rangle} e^{2 \pi i \langle x, y \rangle} dx\, dy
        \\
        & = a \cdot b \int_{\mathbb{R}^n} e^{2 \pi i \langle y, q \rangle} \delta_{y - \Theta(p)} \,dy
        \\
        & = e^{2 \pi i \langle \Theta(p), q \rangle} a \cdot b .
    \end{split}
\end{equation*}
Thus, given $a \in \mathcal A_p$ and $b \in \mathcal A_q$ one has
\begin{equation} \label{homogeneous_rieff_product}
    a \cdot_\Theta b = e^{2 \pi i \langle \Theta(p), q \rangle}a \cdot b .
\end{equation}

\begin{remark}
Notice that $\mathcal A_\Theta$ also possesses a $\mathbb{Z}^n$-grading such that $(\mathcal A_\Theta)_p = \mathcal A_p$ for every $p \in \mathbb{Z}^n$. Due to \eqref{homogeneous_rieff_product}, we have $a\cdot_\Theta b = a\cdot b$ for any $a,b\in \mathcal A_{\pm p}$, $p\in\mathbb{Z}^n$. Indeed, for any skew
symmetric $\Theta\in M_n(\mathbb R^n)$ and $p\in\mathbb Z^n$, one has $\langle\Theta\, p\, ,\, \pm p\rangle=0$. The involution on $(\mathcal A_\Theta)_p$ coincides with the involution on $\mathcal A_p$.
\end{remark}

Consider a $C^*$-dynamical system $(\mathcal A,\mathbb T^n, \alpha)$, and its covariant
representation $(\pi,U)$ on a Hilbert space $\mathcal H$. For any $p\in \mathbb{Z}^n\simeq  \mathbb{\widehat{T}}^n$, put
\[
 \mathcal H_p = \{ h\in \mathcal H \mid U_t h = e^{2\pi i \langle t,p\rangle} h\}.
\]
Then $\mathcal H = \bigoplus _{p\in \mathbb Z^n} \mathcal H_p$ (see \cite{williams}).

\begin{proposition}[\cite{warped}, Theorem 2.8]\label{theta_rep}
Let $(\pi,U)$ be a covariant representation of $(\mathcal A,\mathbb T^n,\alpha)$ on a Hilbert space $\mathcal{H}$. Then one can define a representation $\pi_\Theta$ of $\mathcal A_\Theta$ as follows:
\[
\pi_\Theta(a) \xi = e^{2 \pi i \langle \Theta(p), q \rangle} \pi(a) \xi,
\]
for every $\xi \in \mathcal{H}_q$, $a \in \mathcal A_p$, $p,q\in \mathbb Z^n$.
Moreover, $\pi_\Theta$ is faithful if and only if $\pi$ is faithful.
\end{proposition}

 It is known that Rieffel's deformation can be embedded into $M(\mathcal A \rtimes_\alpha \mathbb{R}^n)$, but for the periodic actions we have an explicit description of this embedding.

 \begin{proposition}[\cite{Amandip}, Lemma 3.1.1]
The following mapping defines an embedding
\[ 
i_{\mathcal A_\Theta} : \mathcal A_\Theta \rightarrow M(\mathcal A \rtimes_\alpha \mathbb R^n), \ i_{\mathcal A_\Theta}(a_p) = i_{\mathcal A}(a_p)i_{\mathbb R^n}(-\Theta(p)), 
\]
where $p \in \mathbb{Z}^n$ and $a_p$ is homogeneous of degree $p$.
\end{proposition}

\begin{proposition}[\cite{Kasprzak_rieffel}, Proposition 3.2 and \cite{Amandip}, Section 3.1]\label{prop4}
Let $(\mathcal A,\mathbb R^n,\alpha)$ be a $C^*$\nobreakdash-dyna\-mical system with  periodic $\alpha$ and unital $\mathcal A$.  Put $\mathcal A_{\Theta}$ to be the Rieffel deformation of $\mathcal A$. There exist a periodic
action $\alpha^{\Theta}$ of $\mathbb R^n$ on $\mathcal A_{\Theta}$ and an isomorphism $\Psi : \mathcal A_\Theta \rtimes_{\alpha^\Theta} \mathbb{R}^n \rightarrow \mathcal A \rtimes_\alpha \mathbb{R}^n$ such that the following diagram is commutative
\[ \begin{tikzcd}
& C^*(\mathbb{R}^n) \simeq C_0(\mathbb{R}^n) \arrow[dl, "i_{\mathbb{R}^n}"'] \arrow[dr, "i_{\mathbb{R}^n}"] & \\
\mathcal A_\Theta \rtimes_{\alpha^\Theta} \mathbb{R}^n \arrow[rr, "\Psi"] & & \mathcal A \rtimes_\alpha \mathbb{R}^n
\end{tikzcd} \]
\end{proposition}

Namely, $\alpha^{\Theta}(a)=\alpha(a)$ holds for any $a\in \mathcal A_p$, $p\in\mathbb{Z}^n$. Then it is easy to verify that $i_{\mathcal A_{\Theta}}\colon \mathcal A_{\Theta}\rightarrow M(\mathcal A\rtimes_{\alpha}\mathbb{R}^n)$ with $i_{\mathbb R^n}\colon\mathbb R^n\rightarrow M(\mathcal A\rtimes_\alpha \mathbb R^n)$ determine a covariant representation of $(\mathcal A_{\Theta},\mathbb{R}^n,\alpha^{\Theta})$ in $M(
\mathcal A\rtimes_{\alpha}\mathbb R^n)$. Hence, by the universal property of crossed product we get the corresponding homomorphism
\[
\Psi\colon \mathcal A_{\Theta}\rtimes_{\alpha^\Theta}\mathbb{R}^n\rightarrow M(\mathcal A\rtimes_\alpha \mathbb R^n).
\]
In fact, the range of $\Psi$ coincides with $\mathcal A\rtimes_\alpha \mathbb R^n$ and $\Psi$ defines an isomorphism
\begin{equation}\label{PsiHom}
    \Psi\colon \mathcal A_{\Theta}\rtimes_{\alpha^\Theta}\mathbb{R}^n\rightarrow \mathcal A\rtimes_\alpha \mathbb R^n,
\end{equation}
see \cite{Kasprzak_rieffel,Amandip} for more detailed considerations.

The following propositions shows that Rieffel's deformation inherits properties of the non-deformed counterpart.
\begin{proposition}[\cite{Kasprzak_rieffel}, Theorem 3.10]\label{Rieff_nuclear}
A $C^*$-algebra $\mathcal A_\Theta$ is nuclear if and only if $\mathcal A$ is nuclear.
\end{proposition}

\begin{proposition}[\cite{Kasprzak_rieffel}, Theorem 3.13]\label{Rieff_K_theory}
For a $C^*$-algebra $\mathcal A$ one has
\[
K_0(\mathcal A_\Theta) = K_0(\mathcal A)\quad \mbox{and}\quad K_1(\mathcal A_{\Theta})=K_1(\mathcal A).
\]
\end{proposition}

\subsubsection{Rieffel's deformation of a tensor product }
In this part we apply Rieffel's deformation procedure to a tensor product of two nuclear unital $C^*$-algebras equipped with an action of $\mathbb{T}$.

Let $\mathcal A$, $\mathcal B$ be $C^*$-algebras with actions $\alpha$ and $\beta$ of $\mathbb{T}$. Then there is a natural action $\alpha \otimes \beta$ of $\mathbb{T}^2$ on $\mathcal A \otimes \mathcal B$ defined as
\[ (\alpha \otimes \beta)_{\varphi_1, \varphi_2}(a \otimes b) = \alpha_{\varphi_1}(a) \otimes \beta_{\varphi_2}(b). \]
Consider the induced gradings on $\mathcal A$ and $\mathcal B$:
\[ \mathcal A = \bigoplus_{p_1 \in \mathbb{Z}} \mathcal A_{p_1}, \quad \mathcal B = \bigoplus_{p_2 \in \mathbb{Z}} \mathcal B_{p_2}. \]
Then the corresponding grading on $\mathcal A \otimes \mathcal B$ is
\[ \mathcal A \otimes \mathcal B := \bigoplus_{(p_1, p_2)^t \in \mathbb{Z}^2} \mathcal A_{p_1} \otimes \mathcal B_{p_2}. \]
In particular, $a \otimes \mathbf 1 \in (\mathcal A \otimes \mathcal B)_{(p_1, 0)^t}$ and $\mathbf 1 \otimes b \in (\mathcal A \otimes  \mathcal B)_{(0, p_2)^t}$, where $a \in \mathcal A_{p_1}$ and $b \in \mathcal  B_{p_2}$.

Given $q = e^{2 \pi i \varphi_0}$, consider
\begin{equation}\label{theta_q}
\Theta_q = \left( \begin{array}{cc}
    0 & \frac{\varphi_0}{2}  \\
    -\frac{\varphi_0}{2} & 0
\end{array} \right).
\end{equation}

 One can see that 
 the Rieffel deformation $(\mathcal A \otimes \mathcal B)_{\Theta_q}$ is in fact the universal $C^*$-algebra $\mathcal A\otimes_{\Theta_q} \mathcal B$, generated by all homogeneous elements
 \[
 a \in \mathcal A_{p_1},\ b\in \mathcal B_{p_2}, \quad p_1, p_2 \in \mathbb Z,
 \]
 subject to the relations
 \begin{equation}\label{r_rel}
 ba = e^{2 \pi i \varphi_0 p_1 p_2} ab.
 \end{equation}
 Below we present more precise formulation and elementary prove of this result. The discussion on universality properties for much more general deformation of tensor product can be found in \cite{mey_wor}.
 
First, construct $*$-algebra $\mathcal A\hat{\otimes}_{\Theta_q}\mathcal B$,
\[
\mathcal A\hat{\otimes}_{\Theta_q}\mathcal B=\mathbb C \langle\, 
a\in\mathcal A_{p_1},\ b\in\mathcal B_{p_2}\mid ba = e^{2 \pi i \varphi_0 p_1 p_2} ab,\ p_1,p_2\in\mathbb Z \,
\rangle
\]
It is easy to see that correspondence $\tilde{\eta}$, determined by
\[
\tilde{\eta}(a)=a\otimes\mathbf 1,\quad 
\tilde{\eta}(b)=\mathbf 1\otimes b,\quad a\in\mathcal A_{p_1},\ b\in\mathcal B_{p_2},\ p_1,p_2\in\mathbb Z,
\]
extends to  a $*$-algebra homomorphism $\tilde{\eta}\colon\mathcal A\hat{\otimes}_{\Theta_q}\mathcal B\rightarrow (\mathcal A\otimes\mathcal B)_{\Theta_q}$. Indeed, let $e_1=(1,0)^t$, $e_2=(0,1)^t$. Then for $a\in\mathcal A_{p_1}$, $b\in\mathcal B_{p_2}$ one has
\[
\tilde{\eta}(a)\cdot_{\Theta_q}\tilde{\eta}(b)=(a\otimes\mathbf 1)\cdot_{\Theta_q}(\mathbf 1\otimes b)=e^{2\pi i\varphi_0 \langle p_1\Theta_q\,  e_1\, ,\, p_2 e_2\rangle} a\otimes b=e^{-\pi i\varphi_0 p_1 p_2} a\otimes b
\]
and
\[
\tilde{\eta}(b)\cdot_{\Theta_q}\tilde{\eta}(a)=(\mathbf 1\otimes b)\cdot_{\Theta_q}(a\otimes\mathbf 1)=e^{2\pi i\varphi_0 \langle p_2\Theta_q\,  e_2\, ,\, p_1 e_1\rangle} a\otimes b=e^{\pi i\varphi_0 p_1 p_2} a\otimes b.
\]
Hence
\[
\tilde{\eta}(b)\cdot_{\Theta_q}\tilde{\eta}(a)=e^{2\pi i\varphi_o p_1p_2}\tilde{\eta}(a)\cdot_{\Theta_q}\tilde{\eta}(b).
\]
In particular
\[
\tilde{\eta} (ab)=e^{-\pi\varphi_0 p_1p_2}a\otimes b,\quad a\in\mathcal A_{p_1},\ b\in\mathcal B_{p_2},\ p_1,p_2\in\mathbb Z.
\]

Since $(\mathcal A\otimes\mathcal B)_{\Theta_q}$ is a $C^*$-algebra, the set $\mathsf{Rep}\, \mathcal A\hat{\otimes}_{\Theta_q}\mathcal B $ is non-empty. Further, take any $\pi\in\mathsf{Rep}\, \mathcal A\hat{\otimes}_{\Theta_q}\mathcal B$. Consider its restriction, $\pi_{|\mathcal A}$ to $\mathcal A$. Then for any $a\in\mathcal A_{p_1}$ one has
\[
\|\pi(a)\|=\|\pi_{|\mathcal A} (a)\|\le \| a\|_{\mathcal A}.
\]
Analogously, $\|\pi(b)\|\le \| b\|_{\mathcal B}$, $b\in B_{p_2}$. Hence for any $x\in \mathcal A\hat{\otimes}_{\Theta_q}\mathcal B$ one has 
\[
\sup_{\pi\in\mathsf{Rep}\, \mathcal A\hat{\otimes}_{\Theta_q}\mathcal B} \|\pi(x)\|<\infty
\]
and the universal $C^*$-algebra $\mathcal A \otimes_{\Theta_q}\mathcal B:=C^* (\mathcal A\hat{\otimes}_{\Theta_q}\mathcal B)$ exists.

Notice that one has the natural $\mathbb T^2$ action $\gamma$ on $\mathcal A\otimes_{\Theta_q}\mathcal B$ determined by
\[
\gamma_{\varphi_1,\varphi_2} (a\cdot b)=e^{2\pi i (p_1\varphi_1+p_2\varphi_2)} a\cdot b,\quad a\in\mathcal A_{p_1},\ b\in\mathcal{B}_{p_2}, p_1,p_2\in\mathbb Z.
\]

Proposition \ref{double_deformation} implies the following result.

\begin{theorem}\label{Rieffel_universal_iso}
One has a $\mathbb{T}^2$-equivariant $*$-isomorphism
\[ 
\mathcal A\otimes_{\Theta_q} \mathcal B
 \simeq (\mathcal  A \otimes \mathcal B)_{\Theta_q}. 
 \]
\end{theorem}

\begin{proof}
The following argument is due to a private communication by P.~Kaszprak.

By the universal property we extend the homomorphism
\[
\tilde{\eta}\colon\mathcal A\hat{\otimes}_{\Theta_q}\mathcal B\rightarrow (\mathcal A\otimes\mathcal B)_{\Theta_q}
\]
to surjective homomorphism
\[
\eta\colon\mathcal A\otimes_{\Theta_q}\mathcal B\rightarrow (\mathcal A\otimes\mathcal B)_{\Theta_q}.
\]
Evidently, $\eta$ is equivariant with respect to $\mathbb T^2$ actions on $\mathcal A\otimes_{\Theta_q}\mathcal B$ and $(\mathcal A\otimes\mathcal B)_{\Theta_q}$ described above.

To show the injectivity of $\eta$, we construct homomorphism
\[ 
\psi \colon \mathcal A \otimes \mathcal B \to (\mathcal A\otimes_{\Theta_q} \mathcal B)_{-\Theta_q}, 
\]
defined as follows. For homogeneous $a \in \mathcal A_{p_1}$, $b \in \mathcal B_{p_2}$, put
\[
\psi(a\otimes \mathbf 1) = a \in (\mathcal A\otimes_{\Theta_q} \mathcal B)_{-\Theta_q}, \quad \psi(\mathbf 1\otimes b) = b \in (\mathcal A\otimes_{\Theta_q} \mathcal B)_{-\Theta_q}.
\]
If $a \in \mathcal A_{p_1}$ and $b \in \mathcal B_{p_2}$, then
\begin{equation*}
    \psi( a \otimes \mathbf 1) \cdot_{-\Theta_q} \psi(\mathbf 1\otimes b ) = e^{ \pi i \varphi_0p_1p_2} ab  =   e^{-\pi i \varphi_0 p_1p_2} ba = \psi(\mathbf 1\otimes b ) \cdot_{-\Theta_q} \psi(a \otimes \mathbf 1 ).
\end{equation*}
Here, we use the relation $ab = e^{-2\pi i \varphi_0 p_1p_2}ba$ which holds in $\mathcal A\otimes_{\Theta_q} \mathcal B$.

Due to the universal property of tensor product, $\psi$ extends to a surjective homomorphism from $\mathcal A \otimes \mathcal B$ to $(\mathcal A\otimes_{\Theta_q} \mathcal B)_{-\Theta_q}$. 

Recall that the equivariant homomorphism
\[
\eta\colon \mathcal A\otimes_{\Theta_q} \mathcal B \to (\mathcal  A \otimes \mathcal B)_{\Theta_q}
\] 
is injective if and only if the induced homomorphism
\[
\eta_{-\Theta_q}\colon (\mathcal A\otimes_{\Theta_q} \mathcal B)_{-\Theta_q} \to ((\mathcal  A \otimes \mathcal B)_{\Theta_q})_{-\Theta_q}
\]
determined by
\[
\eta_{-\Theta_q}(a)=a\otimes\mathbf 1,\quad \eta_{-\Theta_q}(b)=\mathbf 1\otimes b,\quad a\in\mathcal A_{p_1},\ b\in\mathcal B_{p_2},\ p_1,p_2\in\mathbb Z_{+},
\]
is injective. Recall also that, due to functorial properties of Rieffel deformation the identity mapping 
\[
a\otimes b\mapsto a\otimes b,\quad a\in\mathcal A_{p_1},\ b\in\mathcal B_{p_2},\quad
 p_1,p_2\in\mathbb Z,
\]
extends to isomorphism $((\mathcal  A \otimes \mathcal B)_{\Theta_q})_{-\Theta_q}\simeq \mathcal A\otimes\mathcal B$.  

So, one has the following commutative diagram
\[
\begin{tikzcd}
(\mathcal A\otimes_{\Theta_q} \mathcal B )_{-\Theta_q} 
\arrow[r, "\eta_{-\Theta_q}"] & ((\mathcal A \otimes \mathcal B)_{\Theta_q})_{-\Theta_q} \\
\mathcal A \otimes \mathcal B \arrow[u, "\psi"] \arrow[ur, equal]
\end{tikzcd}
\]
Since $\psi$ is surjective, $\eta_{-\Theta_q}$ is injective.
\end{proof}

\subsection{Nuclearity of $\mathcal{E}_{n,m}^q$}

The nuclearity of $\mathcal O_n^{(0)}\otimes\mathcal O_m^{(0)}$ and Proposition \ref{Rieff_nuclear} immediately imply the following

\begin{corollary}\label{EIsNuclear}
 The $C^*$-algebra $\mathcal{E}_{n,m}^q$ is nuclear for any $q\in\mathbb C$, $|q|=1$.
\end{corollary}

The nuclearity of $\mathcal E_{n,m}^q$ can also be shown using  more explicit arguments. One can use the standard trick of untwisting the $q$-deformation in the crossed product, which clarifies  informally the nature of isomorphism (\ref{PsiHom}). Namely, for $q = e^{2 \pi i \varphi_0}$
consider the action $\alpha_q$ of $\mathbb{Z}$ on $\mathcal E_{n,m}^q$ defined on the generators as
\[
\alpha_q^k(s_j)=e^{\pi i k\varphi_0} s_j,\quad \alpha_q^k(t_r)=e^{-\pi i k\varphi_0} t_r,\quad
j=1,\dots,n,\ r=1,\dots,m,\ k\in\mathbb Z.
\]
Denote by the same symbol the similar action on
$\mathcal E_{n,m}^1\simeq \mathcal{O}_n^{(0)}\otimes\mathcal{O}_m^{(0)}$.
Here we denote by $\widetilde{s}_j$ and $\widetilde{t}_r$ the generators of $\mathcal E_{n,m}^1$.

\begin{proposition}
 For any  $\varphi_0\in [0,1)$, one has an isomorphism
 $\mathcal E_{n,m}^q\rtimes_{\alpha_q} \mathbb Z\simeq \mathcal E_{n,m}^1\rtimes_{\alpha_q} \mathbb Z$.
\end{proposition}
\begin{proof}
 Recall that $\mathcal E_{n,m}^1\rtimes_{\alpha_q} \mathbb Z$ is generated as a $C^*$-algebra by elements
 $\widetilde{s}_j$, $\widetilde{t}_r$ and a unitary $u$, such that the following relations satisfied
 \[
   u \widetilde{s}_ju^*=e^{i\pi \varphi_0}\, \widetilde{s}_j,\quad
    u \widetilde{t}_ru^*=e^{-i\pi \varphi_0}\, \widetilde{t}_r,\quad j=\overline{1,n},\ r=\overline{1,m}.
 \]
Put $\widehat{s}_j=\widetilde{s}_j\, u$ and $\widehat{t}_r=\widetilde{t}_r\, u$. Obviously,
$\widehat{s}_j$, $\widehat{t}_r$ and $u$ generate $\mathcal{E}_{n,m}^1\rtimes_{\alpha_q} \mathbb Z$. Further,
\[
 \widehat{s}_j^*\widehat{s}_k=\delta_{jk}\mathbf 1,\quad \widehat{t}_r^*\widehat{t}_l=\delta_{rl}\mathbf{1}
\]
and
\[
\widehat{s}_j\widehat{t}_r=\widetilde{s}_j u\widetilde{t}_r u=e^{-i\pi\varphi_0}\widetilde{s}_j\widetilde{t}_r u^2=
e^{-i\pi\varphi_0}\widetilde{t}_r\widetilde{s}_j u^2=
e^{-2\pi i\varphi_0}\widetilde{t}_r u\widetilde{s}_j u=\overline{q}\, \widehat{s}_j\widehat{t}_r.
\]
In a similar way we get $\widehat{s}_j^*\widehat{t}_r= q\widehat{t}_r\widehat{s}_j^*$, $j=\overline{1,n}$, $r=\overline{1,m}$.
Finally
\[
u\widehat{s}_j u^*=e^{i\pi\varphi_0}\widehat{s}_j,\quad u\widehat{t_r}u^*=e^{-i\pi\varphi_0}\widehat{t}_r.
\]
Hence the correspondence
\[
 s_j\mapsto \widehat{s}_j,\quad t_j\mapsto \widehat{t}_j,\quad u\mapsto u,
\]
determines a homomorphism
$\Phi_q\colon \mathcal E_{n,m}^q\rtimes_{\alpha_q} \mathbb Z \rightarrow \mathcal E_{n,m}^1\rtimes_{\alpha_q} \mathbb Z$.
The inverse is constructed evidently.
 \end{proof}

Let us show the nuclearity of $\mathcal{E}_{n,m}^q$ again. Indeed, $\mathcal{E}_{n,m}^1=\mathcal O_{n}^{(0)}\otimes\mathcal{O}_m^{(0)}$ is nuclear. Then so is the crossed product
 $\mathcal E_{n,m}^1\rtimes_{\alpha_q} \mathbb Z$. Then due to the isomorphism above,
 $\mathcal E_{n,m}^q\rtimes_{\alpha_q} \mathbb Z$ is nuclear, implying the nuclearity of $\mathcal E_{n,m}^q$, see
 \cite{Black}.

\subsection{Fock representation of $\mathcal E_{n,m}^q$}
In this part we study the Fock representation of $\mathcal E_{n,m}^q$.

First of all let us stress out that according to Theorem \ref{Rieffel_universal_iso} we can identify $\mathcal E_{n,m}^q$ with $(\mathcal O_n^{(0)}\otimes\mathcal O_m^{(0)})_{\Theta_q}$. In particular we use this isomorphism below to show that Fock representation of $\mathcal E_{n,m}^q$ is faithful.

\begin{definition}
The Fock representation of $\mathcal E_{n,m}^q$ is the unique up to unitary equivalence irreducible $*$-representation $\pi_F^q$ determined by the action  on vacuum vector $\Omega$, $||\Omega||=1$,
\[
\pi_F^q(s_j^*)\Omega = 0,\quad \pi_F^q(t_r^*)\Omega =0,\quad j=\overline{1,n},\ r=\overline{1,m}.
\]
\end{definition}

Denote by $\pi_{F,n}$ the Fock representation of $\mathcal{O}_n^{(0)}\subset\mathcal E_{n,m}^q$ acting on the space
\[
\mathcal{F}_n=\mathcal{T}(\mathcal{H}_n)= \mathbb{C}\Omega\oplus\bigoplus_{d=1}^{\infty} \mathcal{H}_n^{\otimes d},\quad \mathcal{H}_n=\mathbb{C}^n,
\]
by formulas
\begin{align*}
\pi_{F,n}(s_j)\Omega& =e_j,\quad \pi_{F,n}(s_j)e_{i_1}\otimes e_{i_2}\cdots\otimes e_{i_d}=e_j\otimes e_{i_1}\otimes e_{i_2}\cdots\otimes e_{i_d},\\
\pi_{F,n}(s_j^*)\Omega& =0,\quad \pi_{F,n}(s_j^*)e_{i_1}\otimes e_{i_2}\otimes\cdots\otimes e_{i_d}=\delta_{ji_1}e_{i_2}\otimes\cdots\otimes e_{i_d},\quad d\in\mathbb{N},
\end{align*}
where $e_1$, \dots, $e_n$ is the standard orthonormal basis of $\mathcal{H}_n$. Notice that $\pi_{F,n}$ is the unique irreducible faithful representations of $\mathcal{O}_n^{(0)}$, see for example \cite{jsw2}.

Below we give an explicit formula for $\pi_F^q (s_j)$,
$\pi_F^q (t_r)$. Consider the Fock representations $\pi_{F,n}$ and $\pi_{F,m}$ of $*$-subalgebras $C^*(\{s_1, \ldots, s_n\})=\mathcal O_n^{(0)}\subset\mathcal E_{n,m}^q$ and $C^*(\{t_1, \ldots, t_m\})=\mathcal{O}_m^{(0)}\subset \mathcal E_{n,m}^q$ respectively. Denote by $\Omega_n\in\mathcal F_n$ and $\Omega_m\in\mathcal F_m$ the corresponding vacuum vectors.
\begin{theorem}
The Fock representation $\pi_F^q$ of $\mathcal E_{n,m}^q$ acts on the space $\mathcal{F}=\mathcal{F}_n\otimes\mathcal{F}_m$ as follows
\begin{align*}
\pi_F^q(s_j)&=\pi_{F,n} (s_j)\otimes d_m(q^{-\frac{1}{2}}),\quad j = \overline{1,n},\\
\pi_F^q(t_r)&= d_n(q^{\frac{1}{2}})\otimes \pi_{F,m} (t_r),\quad r=\overline{1,m},
\end{align*}
where $d_k(\lambda)$ acts on $\mathcal F_k$, $k=n,m$ by
\[
d_k(\lambda)\Omega_k =\Omega_k,\quad d_k(\lambda)X=\lambda^l X,\quad X\in \mathcal{H}_k^{\otimes l},\quad l\in\mathbb{N}.
\]
\end{theorem}
\begin{proof}
It is a direct calculation to verify that the operators defined above satisfy the relations of $\mathcal E_{n,m}^q$. Since $\pi_{F,k}$ is irreducible on $\mathcal{F}_k$, $k=m,n$, the representation $\pi_F^q$ is irreducible on $\mathcal{F}_n\otimes\mathcal{F}_m$. Finally put $\Omega=\Omega_n\otimes\Omega_m$, then obviously
\[
\pi_F^q(s_j^*)\Omega=0,\ \mbox{and}\ \pi_F^q(t_r^*)\Omega=0,\quad j=\overline{1,n}, r=\overline{1,m}
\]
Thus $\pi_F^q$ is the Fock representation of $\mathcal E_{n,m}^q$.
 \end{proof}
\begin{remark}\label{rem_fock}
In some cases, it will be more convenient to present the operators of the Fock representation of $\mathcal E_{n,m}^q$ in one of the alternative forms,
\begin{align*}
\pi_F^q(s_j)&=\pi_{F,n} (s_j)\otimes\mathbf{1}_{\mathcal F_m},\quad j = \overline{1,n},
\\
\pi_F^q(t_r)&= d_n(q)\otimes \pi_{F,m} (t_r),\quad r=\overline{1,m},
\end{align*}
or
\begin{align*}
\pi_F^q(s_j)&=\pi_{F,n} (s_j)\otimes d_m(q^{-1}),\quad j = \overline{1,n},\\
\pi_F^q(t_r)&= \mathbf{1}_{\mathcal F_n}\otimes \pi_{F,m} (t_r),\quad r=\overline{1,m},
\end{align*}
which are obviously unitary equivalent to the one presented in the statement above.
\end{remark}
The next step we show that in fact $\pi_F^q$ can be obtained applying the  construction form Proposition \ref{theta_rep}  to the representation $\pi_{F,n}\otimes\pi_{F,m}$ of $\mathcal O_n^{(0)}\otimes\mathcal O_m^{(0)}$. 
To this end construct the family of unitary operators $\{U_{\varphi_1,\varphi_2},\ \varphi_1,\varphi_2\in [0,2\pi)\}$ acting on ${\mathcal F_n\otimes\mathcal F_m}$ as follows
\[
U_{\varphi_1,\varphi_2}(\xi_1\otimes\xi_2)=
e^{2\pi i (\varphi_1 p_1+\varphi_2 p_2)}\, \xi_1\otimes\xi_2,\quad \xi_1\in
\mathcal H_n^{\otimes p_1},\ \xi_2\in\mathcal H_m^{\otimes p_2},\quad p_1, p_2\in\mathbb Z_{+}.
\]
The pair $(\pi_{F,n} \otimes \pi_{F,m},\ U_{\varphi_1,\varphi_2})$ determines a covariant representation of $(\mathcal{O}_n^{(0)} \otimes \mathcal{O}_m^{(0)},\mathbb{T}^2,\alpha)$, where as above
\[
\alpha_{\varphi_1,\varphi_2}(s_j\otimes\mathbf 1)=e^{2\pi i\varphi_1} (s_j\otimes\mathbf 1),\quad \alpha_{\varphi_1,\varphi_2}(\mathbf 1\otimes t_r)=
e^{2\pi i\varphi_2} (\mathbf 1\otimes t_r).
\]
Notice that for $p=(p_1,p_2)^t\in\mathbb{Z}_{+}^2$, the subspace $\mathcal{H}_n^{\otimes p_1}\otimes \mathcal{H}_m^{\otimes p_2}$ is the $(p_1,p_2)^t$-homogeneous component of $\mathcal F$ related to the action of $U_{\varphi_1,\varphi_2}$, and $(\mathcal F)_p=\{0\}$ for any $p\in\mathbb Z^2\setminus\mathbb Z_{+}^2 $.

Recall also that $\widehat{s}_j=s_j\otimes \mathbf{1}$ is contained in $e_1=(1,0)^t$-homogeneous component and $\widehat{t}_r=\mathbf 1\otimes t_r$ is in $e_2=(0,1)^t$-homogeneous component of $\mathcal O_n^{(0)}\otimes\mathcal O_m^{(0)}$ with respect to $\alpha$. Now one can apply Proposition \ref{theta_rep}. Namely, given $\xi=\xi_1\otimes\xi_2 \in \mathcal{H}_n^{\otimes p_1} \otimes \mathcal{H}_m^{\otimes p_2}$ one gets
\begin{align*}
(\pi_{F,n} \otimes \pi_{F,m})_{\Theta_q}(\widehat{s}_j)\, \xi &= e^{2 \pi i \langle\Theta_q\, e_1,\, p  \rangle }\, \pi_{F,n} \otimes \pi_{F,m}(\widehat{s}_j)\, \xi = \\
&= \pi_{F,n}(s_j)\xi_1 \otimes e^{-\pi i\, p_2\, \varphi_0}\, \xi_2=(\pi_{F,n}(s_j)\otimes d_m(q^{-\frac{1}{2}}))\, \xi,
\end{align*}
and
\begin{align*}
(\pi_{F,n} \otimes \pi_{F,m})_{\Theta_q}(\widehat{t}_r)\, \xi &= e^{2 \pi i \langle\Theta_q\, e_2,\, p  \rangle }\, \pi_{F,n} \otimes \pi_{F,m}(\widehat{t}_r)\xi = \\
&=e^{\pi i\, p_1\, \varphi_0}\, \xi_1 \otimes
\pi_{F,m}(t_r)\, \xi_2= ( d_n(q^{\frac{1}{2}})\otimes \pi_{F,m}(t_r))\,\xi.
\end{align*}
Notice that for any $j=\overline{1,n}$, and $r=\overline{1,m}$,
\[
(\pi_{F,n}\otimes \pi_{F,m})_{\Theta_q}(\widehat{s}_j^*)\Omega=0,\quad
(\pi_{F,n}\otimes \pi_{F,m})_{\Theta_q}(\widehat{t}_r^*)\Omega=0.
\]
So, we have shown that $\pi_F^q=(\pi_{F,n}\otimes \pi_{F,m})_{\Theta_q}$. 

Since $\pi_{F,n}\otimes\pi_{F,m}$ is faithful representation of $\mathcal O_n^{(0)}\otimes\mathcal O_m^{(0)}$ we immediately get the following result.
\begin{theorem}
The Fock representation $\pi_F^q$ of $\mathcal E_{m,n}^q$ is faithful.
\end{theorem}

We finish this part by an analog of the well-known Wold decomposition theorem for a single isometry. Recall that
\[ Q = \sum_{j = 1}^n s_j s_j^*, \ P = \sum_{r = 1}^m t_r t_r^*. \]
\begin{theorem}[Generalised Wold decomposition]\label{wold_dec}
Let $\pi\colon\mathcal E_{n,m}^q\rightarrow \mathbb{B}(\mathcal H)$ be a $*$-repre\-sen\-ta\-tion. Then
\[
\mathcal H=\mathcal H_1\oplus\mathcal H_2\oplus\mathcal H_3\oplus\mathcal H_4,
\]
where each $\mathcal H_j$, $j=1,2,3,4$, is invariant with respect to $\pi$, and for
$\pi_j=\pi\restriction_{\mathcal H_j}$ one has
\begin{itemize}
\item $\mathcal H_1=\mathcal F\otimes\mathcal K$ for some Hilbert space $\mathcal{K}$, and $\pi_1=\pi_F^q\otimes\mathbf{1}_{\mathcal K}$;
\item $\pi_2(\mathbf{1}-Q)=0$, $\pi_2(\mathbf 1 -P)\ne 0$;
\item  $\pi_3(\mathbf{1}-P)=0$, $\pi_3(\mathbf 1 -Q)\ne 0$;
\item $\pi_4 (\mathbf{1}-Q)=0$, $\pi_4 (\mathbf 1- P)=0$;
\end{itemize}
where any of $\mathcal H_j$, $j=1,2,3,4$, could be zero.
\end{theorem}

\begin{proof}
We will use the fact that any representation of $\mathcal{O}_n^{(0)}$ is a direct sum of a multiple of the Fock representation and a representation of $\mathcal O_n$.

So, restrict $\pi$ to $\mathcal{O}_n^{(0)}\subset\mathcal E_{n,m}^q$, and decompose $\mathcal H=\mathcal H_{F}\oplus\mathcal H_F^{\perp}$, where
\[
\pi(\mathbf 1-Q)_{|\mathcal H_F^{\perp}}=0,
\]
and $\pi(\mathcal{O}_n^{(0)})_{|\mathcal{H}_F}$ is a multiple of the Fock representation.
Denote
\[
S_j := \pi(s_j)\restriction_{\mathcal H_F}, \quad Q := \pi(Q)\restriction_{\mathcal H_F}.
\]
Put  $S_{\emptyset}:=\mathbf{1}_{\mathcal{H}_F}$ and $S_\lambda: = S_{\lambda_1}\cdots S_{\lambda_k}$ for any non-empty $\Lambda_n\ni\lambda=(\lambda_1,\ldots,\lambda_k)$ .
Since
\[
\mathcal H_F=\bigoplus_{\lambda\in\Lambda_n} S_{\lambda} (\ker Q),
\]
it is invariant with respect to $\pi(t_r)$, $\pi(t_r^*)$, $r=\overline{1,m}$. Indeed, $t_r Q=Q t_r$ in $\mathcal E_{n,m}^q$, implying the invariance of $\ker Q$ with respect to $\pi(t_r)$ and $\pi(t_r^*)$. Denote $\ker Q$ by $\mathcal{G}$ and $T_r := \pi(t_r)\restriction_{\mathcal G}$. Then
\[
\pi(t_r)S_{\lambda}\xi=q^{|\lambda|}S_{\lambda}\pi(t_r)\xi=q^{|\lambda|} S_{\lambda}T_r \xi, \quad \xi \in \mathcal G.
\]
Thus $\mathcal{H}_F\simeq\mathcal{F}_n\otimes\mathcal G$ with
\[
\pi(s_j)\restriction_{\mathcal{H}_F}=\pi_{F,n}(s_j)\otimes\mathbf{1}_{\mathcal G},\quad
\pi(t_r)\restriction_{\mathcal{H}_F}=d_n(q)\otimes T_r,\quad j=\overline{1,n},\ r=\overline{1,m},
\]
where the family $\{T_r\}$ determines a $*$-representation $\widetilde{\pi}$ of $\mathcal O_m^{(0)}$ on $\mathcal G$.

Further, decompose $\mathcal G$ as $\mathcal G=\mathcal G_{F}\oplus\mathcal G_F^{\perp}$ into an orthogonal sum of subspaces invariant with respect to $\widetilde{\pi}$, where $\mathcal G_F=\mathcal F_m\otimes\mathcal{K}$,
\[
\widetilde{\pi}_{\mathcal G_F}(t_r)=\pi_{F,m}(t_r)\otimes\mathbf{1}_{\mathcal{K}},\quad r=\overline{1,m},\quad  \mbox{and}\quad
\widetilde{\pi}\restriction_{\mathcal G_F^{\perp}}(\mathbf{1}-P)=0.
\]
Thus $\mathcal H_F=\left(\mathcal{F}_n\otimes\mathcal F_m\otimes\mathcal K\right)\oplus\left(\mathcal F_n\otimes\mathcal G_F^{\perp}\right)$ and
\begin{align*}
\pi_{\mathcal H_F}(s_j)&=\left(\pi_{F,n}(s_j)\otimes\mathbf{1}_{\mathcal F_m}\otimes\mathbf{1}_K \right)\oplus ( \pi_{F,n}(s_j)\otimes\mathbf{1}_{\mathcal G_F^{\perp}} ),\quad j=\overline{1,n},
\\
\pi_{\mathcal H_F}(t_r)&=\left(d_n(q)\otimes\pi_{F,m}(t_r)\otimes\mathbf{1}_K\right)\oplus
\left(d_n(q)\otimes\widetilde{\pi}_{|\mathcal{G}_F^{\perp}}(t_r)\right),\quad r=\overline{1,m}.
\end{align*}
Put $\mathcal H_1=\mathcal F_n\otimes\mathcal F_m\otimes\mathcal{K}=\mathcal F\otimes\mathcal K$ and notice that that $\pi\restriction_{\mathcal{H}_1}=\pi_F^q\otimes\mathbf{1}_{\mathcal K}$, see Remark \ref{rem_fock}. Put $\mathcal H_3=\mathcal F_n\otimes\mathcal G_F^{\perp}$ and $\pi_3=\pi\restriction_{\mathcal H_3}$ i.e.,
\[
\pi_3(s_j)=\pi_{F,n}(s_j)\otimes\mathbf{1}_{\mathcal{G}_F^\perp},\quad
\pi_3(t_r)=d_n(q)\otimes \widetilde{\pi}_{|\mathcal{G}_F^{\perp}}(t_r),\quad
j=\overline{1,n},\ r=\overline{1,m}.
\]
Evidently, $\pi_3(\mathbf 1 -P)=0$ and $\pi_3(\mathbf 1-Q)\ne 0$.

Finally, applying similar arguments to the invariant subspace $\mathcal H_F^{\perp}$ one can show that there exists a decomposition
\[
\mathcal H_F^{\perp}=\mathcal H_2\oplus \mathcal H_4
\]
into the orthogonal sum of invariant subspaces, where
\begin{itemize}
\item
$\mathcal H_2=\mathcal F_m\otimes\mathcal L$ and
\[
\pi_2(s_j) := \pi\restriction_{\mathcal H_2}(s_j)= d_m (\overline{q})\otimes\widehat{\pi}(s_j),\quad
\pi_2(t_r):= \pi\restriction_{\mathcal H_2}(t_r) = \pi_{F,m}(t_r)\otimes\mathbf 1_{\mathcal L},
\]
for a representation $\widehat{\pi}$ of $\mathcal O_n$. Evidently,
$\pi_2(\mathbf 1 - Q)=0$, $\pi_2 (\mathbf 1 - P)\ne 0$.
\item For $\pi_4:=\pi\restriction_{\mathcal H_4}$ one has
\[
\pi_4(\mathbf 1 -Q)=0,\quad \pi_4(\mathbf 1 - P)=0.
\]
\end{itemize}
 \end{proof}

\subsection{Ideals in $\mathcal{E}_{n,m}^q$}
In this part, we give a complete description of ideals in $\mathcal{E}_{n,m}^q$, and prove their independence on the deformation parameter $q$.

For
\[
 Q=\sum_{j=1}^n s_j s_j^*,\quad P=\sum_{r=1}^m t_r t_r^*.
\]
we consider  two-sided ideals, $\mathcal M_q$ generated by $1 - P$ and $1 - Q$, $\mathcal I_1^q$ generated by $1-Q$, $\mathcal I_2^q$~generated by $\mathbf 1-P$, and $\mathcal I_q$ generated by $(\mathbf 1-Q)(\mathbf 1-P)$. Evidently, \[ \mathcal I_q = \mathcal I^q_1 \cap \mathcal I^q_2 = \mathcal I^q_1 \cdot \mathcal I^q_2. \] Below we will show that any ideal in $\mathcal E_{n,m}^q$ coincides with the one listed above.

To clarify the structure of $\mathcal I_1^q$, $\mathcal I_2^q$ and $\mathcal I_q$, we use the construction of twisted tensor product of a certain $C^*$-algebra with the algebra of compact operators $\mathbb K$, see \cite{weber}. We give a brief review of the construction, adapted to our situation.

Recall that the $C^*$-algebra $\mathbb K$ can be considered as a universal $C^*$-algebra generated by a closed linear span of elements
$e_{\mu\nu}$, $\mu,\nu\in \Lambda_n$ subject to the relations
\[
e_{\mu_1\nu_1}e_{\mu_2\nu_2}=\delta_{\mu_2\nu_1}e_{\mu_1\nu_2},\quad
e_{\mu_1\nu_1}^*=e_{\nu_1 \mu_1},\quad \nu_i,\mu_i\in\Lambda_n,
\]
here $e_{\emptyset}:=e_{\emptyset\emptyset}$ is a minimal projection.

\begin{definition}
Let $\mathcal A$ be a $C^*$-algebra,
\[
\alpha=\{\alpha_{\mu},\ \mu\in\Lambda_n\}\subset \mathsf{Aut}(\mathcal A),\ \mbox{ where}\ \alpha_{\emptyset}=\id_{\mathcal A},
\] and $e_{\mu\nu}$, $\mu,\nu\in\Lambda_n$ be the generators of $\mathbb K$ specified above. Construct the universal $C^*$-algebra
\[
\langle\mathcal A, \mathbb{K} \rangle_\alpha = C^* (a \in\mathcal A, e_{\mu\nu} \in \mathbb{K} \, |\, \ a e_{\mu\nu} = e_{\mu\nu} \alpha_{\nu}^{-1}(\alpha_{\mu} (a)).
\]
We define $\mathcal A \otimes_\alpha \mathbb{K}$ as a subalgebra of $\langle\mathcal A, \mathbb{K} \rangle_\alpha$ generated by $ax$, $a \in\mathcal A \subset \langle\mathcal A, \mathbb{K} \rangle_\alpha$, $x \in \mathbb{K} \subset \langle\mathcal A, \mathbb{K} \rangle_\alpha$.
\end{definition}

Notice that $\langle\mathcal A, \mathbb{K} \rangle_\alpha$ exists for any $C^*$-algebra $\mathcal A$ and family $\alpha\subset\mathsf{Aut} (\mathcal A)$, see \cite{weber}.
\begin{remark}\ \\ \noindent
$1$. Let $x_{\mu}=e_{\mu\emptyset}$. Then $a x_{\mu}=x_{\mu}\alpha_{\mu}(a)$, $a x_{\mu}^*=x_{\mu}^*\alpha_{\mu}^{-1}(a)$, $a\in\mathcal A$, compare with \cite{weber}.\\ \noindent
$2$. For any $a\in\mathcal A$ one has $e_{\mu\nu} a=\alpha_{\mu}^{-1}(\alpha_{\nu}(a))e_{\mu\nu}$ implying that
\[
(a e_{\mu\nu})^*=\alpha_{\mu}^{-1}(\alpha_{\nu}(a))e_{\nu\mu}.
\]  \noindent
$3$. For any $a_1,a_2\in\mathcal A$ one has $(a_1 e_{\mu_1\nu_1})(a_2 e_{\mu_2\nu_2})=\delta_{\nu_1\mu_2} a_1\alpha_{\mu_1}^{-1}(\alpha_{\mu_2}(a_2))e_{\mu_1\nu_2}$.
\end{remark}

\begin{proposition}[\cite{weber}]
Let $\mathcal A$ be a $C^*$-algebra and
\[
\alpha = \{\alpha_{\mu},\ \mu \in \Lambda_n\} \subset \mathsf{Aut}(\mathcal A)\ \mbox{with}\ \alpha_{\emptyset} = \id_{\mathcal A}.
\]
Then the correspondence
\[
a e_{\mu\nu}\mapsto \alpha_{\mu}(a)\otimes e_{\mu\nu},\quad a\in\mathcal A,\ \mu,\nu\in\Lambda_n
\]
extends by linearity and continuity to an isomorphism
\[
\Delta_\alpha\colon\mathcal A\otimes_{\alpha} \mathbb K \rightarrow \mathcal A\otimes \mathbb K,
\]
where $\Delta_{\alpha}^{-1}$ is constructed via the correspondence
\[
a\otimes e_{\mu\nu}\mapsto \alpha_{\mu}^{-1}(a) e_{\mu\nu},\quad a\in \mathcal A, \ \mu,\nu\in\Lambda_n.
\]
\end{proposition}

\begin{remark}
For $x_{\mu}=e_{\mu\emptyset}$, $\mu\in\Lambda_n$ one has, see \cite{weber},
\[
\Delta_{\alpha}(a x_\mu )= \alpha_\mu(a) \otimes x_\mu,\quad
\Delta_{\alpha}(a x_\mu^* )= a \otimes x_\mu^*.
\]
\end{remark}

The following functorial property of $\otimes_{\alpha}\mathbb K$ can be derived easily.
Consider \[ \alpha = (\alpha_\mu)_{\mu \in\Lambda_n} \subset \mathsf{Aut}(\mathcal A), \ \beta = (\beta_\mu)_{\mu \in\Lambda_n} \subset \mathsf{Aut}(\mathcal B). \] Suppose $\varphi :\mathcal A \rightarrow \mathcal B$ is equivariant, i.e.  $\varphi(\alpha_\mu(a)) = \beta_\mu(\varphi(a))$ for any $a\in\mathcal A$ and $\mu\in\Lambda_n$. Then one can define the homomorphism
\[
\varphi \otimes_\alpha^\beta : \mathcal A \otimes_\alpha \mathbb{K} \rightarrow \mathcal B \otimes_\beta \mathbb{K}, \quad \varphi\otimes_\alpha^\beta(ak)= \varphi(a)k,\quad a\in\mathcal A,\ k\in\mathbb K,
\]
making the following diagram commutative
\begin{equation}\label{ktwist_func}
\begin{tikzcd}
\mathcal A \otimes_\alpha \mathbb{K} \arrow[d, "\Delta_\alpha"] \arrow[r, "\varphi \otimes_\alpha^\beta"] & \mathcal B \otimes_\beta \mathbb{K} \arrow[d, "\Delta_\beta"] \\
\mathcal A \otimes \mathbb{K} \arrow[r, "\varphi \otimes \id_\mathbb{K}"] & \mathcal B \otimes \mathbb{K}
\end{tikzcd}
\end{equation}
Namely, it is easy to verify that
\[
(\Delta_\beta^{-1}\circ(\varphi\otimes \id_{\mathbb K})\circ\Delta_\alpha) (a e_{\mu\nu})=\varphi(a)e_{\mu\nu}=\varphi\otimes_{\alpha}^{\beta}(a e_{\mu\nu}),\quad a\in A,\ \mu,\nu\in\Lambda_n.
\]

An important consequence of the commutativity of the diagram above is exactness of the functor $\otimes_{\alpha}\mathbb K$. Let
\[
\beta=(\beta_{\mu})_{\mu\in\Lambda_n}\subset\mathsf{Aut} (\mathcal B),\
\alpha=(\alpha_{\mu})_{\mu\in\Lambda_n}\subset\mathsf{Aut} (\mathcal A),\  \gamma=(\gamma_{\mu})_{\mu\in\Lambda_n}\subset\mathsf{Aut} (\mathcal C)
\]
and consider a short exact sequence
\[
\begin{tikzcd}
 0 \arrow[r] & \mathcal B \arrow[r, "\varphi_1"]  & \mathcal A \arrow[r, "\varphi_2"]&\mathcal C\arrow[r] & 0
 \end{tikzcd}
\]
where $\varphi_1$, $\varphi_2$ are equivariant homomorphisms. Then one has the following short exact sequence
\[
\begin{tikzcd}
 0 \arrow[r] &\mathcal B\otimes_{\beta}\mathbb K \arrow[r, "\varphi_1\otimes_\beta^\alpha"]  & \mathcal A\otimes_{\alpha}\mathbb K \arrow[r, "\varphi_2\otimes_\alpha^\gamma"]&\mathcal C\otimes_\gamma \mathbb K\arrow[r] & 0
 \end{tikzcd}
\]

Now we are ready to study the structure of the ideals $\mathcal I_1^q\, ,\mathcal I_2^q,\mathcal I_q\subset\mathcal{E}_{n,m}^q$. We start with $\mathcal I_1^q$. Notice that
\[
\mathcal I_1^q=c.l.s.\, \{ \,t_{\mu_2}t_{\nu_2}^* s_{\mu_1}(\mathbf 1-Q)s_{\nu_1}^*,\ \mu_1,\nu_1\in\Lambda_n,\ \mu_2,\nu_2\in\Lambda_m\}.
\]
Put $E_{\mu_1\nu_1}=s_{\mu_1}(\mathbf 1-Q)s_{\nu_1}^*$, $\mu_1,\nu_1\in\Lambda_n$. Then $E_{\mu_1\nu_1}$ satisfy the relations for matrix units generating $\mathbb K$.  Moreover, $c.l.s.\,\{E_{\mu\nu},\ \mu,\nu\in\Lambda_n\}$ is an ideal in $\mathcal O_n^{(0)}$ isomorphic to $\mathbb K$.

Consider the family $\alpha^q =(\alpha_\mu)_{ \mu\in\Lambda_n}\subset\mathsf{Aut}(\mathcal O_m^{(0)})$ defined as
\[
\alpha_{\mu} (t_r)=q^{|\mu|} t_r,\quad
\alpha_{\mu} (t_r^*)=q^{-|\mu|} t_r^*,\quad \mu\in\Lambda_n,\ r=\overline{1,m}.
\]
\begin{proposition}\label{str_I1q}
The correspondence $a e_{\mu\nu}\mapsto a E_{\mu\nu}$, $a\in \mathcal O_m^{(0)}$, $\mu,\nu\in\Lambda_n$, extends to an isomorphism
\[
\Delta_{q,1} \colon \mathcal O_m^{(0)}\otimes_{\alpha^q}\mathbb K\rightarrow \mathcal I_1^q.
\]
\end{proposition}
\begin{proof}
We note that for any $\mu_1,\nu_1\in\Lambda_n$ and $\mu_2,\nu_2\in\Lambda_m$ one has
\[
t_{\mu_2}t_{\nu_2}^* E_{\mu_1\nu_1}= q^{(|\nu_1|-|\mu_1|)(|\mu_2|-|\nu_2|)}
E_{\mu_1\nu_1}t_{\mu_2}t_{\nu_2}^* = E_{\mu_1\nu_1}
\alpha_{\nu_1}^{-1}(\alpha_{\mu_1}(t_{\mu_2}t_{\nu_2}^*)).
\]
Thus, due  to the universal property of $\langle \mathcal O_m^{(0)}, \mathbb K \rangle_{\alpha^q}$, the correspondence \[ a e_{\mu \nu} \mapsto a E_{\mu \nu} \] determines a surjective homomorphism
$\Delta_{q,1}\colon \mathcal O_m^{(0)}\otimes_{\alpha^q}\mathbb K\rightarrow \mathcal I_1^q$.

It remains to show that $\Delta_{q,1}$ is injective. Since the Fock representation of $\mathcal E_{n,m}^q$ is faithful, we can identify $\mathcal I_1^q$ with $\pi_F^q (\mathcal I_1^q)$. It will be convenient for us to use the following form of the Fock representation, see Remark \ref{rem_fock},
\begin{align*}
\pi_F^q (s_j)&=\pi_{F,n}(s_j)\otimes\mathbf{1}_{\mathcal F_m}
:= S_j\otimes\mathbf 1_{\mathcal F_m},\ j=\overline{1,n}, \\
\pi_F^q (t_r)&= d_n(q)\otimes\pi_{F,m}(t_r)
:= d_n(q)\otimes T_r,\ r=\overline{1,m}.
\end{align*}
In particular, for any $\mu_1,\nu_1\in\Lambda_n,\ \mu_2,\nu_2\in\Lambda_m$
\[
\pi_F^q (t_{\mu_2}t_{\nu_2}^* E_{\mu_1\nu_1})=
d_n(q^{|\mu_2|-|\nu_2|}) S_{\mu_1} (\mathbf 1 - Q)S_{\nu_1}\otimes T_{\mu_2}T_{\nu_2}^*.
\]

Consider $\Delta_{q,1}\circ\Delta_{\alpha^q}^{-1}\colon \mathcal O_m^{(0)}\otimes\mathbb K\rightarrow \pi_F^q (\mathcal I_1^q)$. We intend to show that
\[
\Delta_{q,1}\circ\Delta_{\alpha^q}^{-1}=\pi_F^1,
\]
where $\pi_F^1$ is the restriction of the Fock representation of $\mathcal{O}_n^{(0)}\otimes\mathcal{O}_m^{(0)}$ to $\mathbb K\otimes\mathcal O_m^{(0)}$, and $\mathbb K$ is generated by $E_{\mu\nu}$ specified above. Notice that the family
\[
\{t_{\mu_2}t_{\nu_2}^*\otimes E_{\mu_1\nu_1},\ \mu_1,\nu_1\in\Lambda_n,\ \mu_2,\nu_2\in\Lambda_m\}
\]
generates $\mathcal{O}_m^{(0)}\otimes\mathbb K$. Then
\[
\Delta_{\alpha^q}^{-1}(t_{\mu_2}t_{\nu_2}^*\otimes E_{\mu_1\nu_1})=
\alpha_{\mu_1}^{-1}(t_{\mu_2}t_{\nu_2}^*) e_{\mu_1\nu_1}=
q^{-|\mu_1|(|\mu_2|-|\nu_2|)}t_{\mu_2}t_{\nu_2}^* e_{\mu_1\nu_1},
\]
and
\begin{align*}
\Delta_{q,1}\circ & \Delta_{\alpha^q}^{-1}(t_{\mu_2}t_{\nu_2}^*\otimes E_{\mu_1\nu_1})=
q^{-|\mu_1|(|\mu_2|-|\nu_2|)}\pi_F^q(t_{\mu_2}t_{\nu_2}^* E_{\mu_1\nu_1})
\\
&=q^{-|\mu_1|(|\mu_2|-|\nu_2|)}d_n(q^{|\mu_2|-|\nu_2|})S_{\mu_1}
(\mathbf 1-Q)S_{\nu_1}^*\otimes T_{\mu_2}T_{\nu_2}^*
\\
&=
q^{-|\mu_1|(|\mu_2|-|\nu_2|)}q^{|\mu_1|(|\mu_2|-|\nu_2|)}S_{\mu_1}d_n(q^{|\mu_2|-|\nu_2|})
(\mathbf 1-Q)S_{\nu_1}^*\otimes T_{\mu_2}T_{\nu_2}^*
\\
&=S_{\mu_1}(\mathbf 1-Q)S_{\nu_1}^*\otimes T_{\mu_2}T_{\nu_2}^*=
\pi_F^1 (E_{\mu_1\nu_1}\otimes t_{\mu_2}t_{\nu_2}^*),
\end{align*}
where we used relations $d_n(\lambda)S_j=\lambda S_j d_n(\lambda)$, $j=\overline{1,n}$, $\lambda\in\mathbb C$, and the obvious fact that
\[
d_n(\lambda) (\mathbf 1 -Q)=\mathbf 1 -Q.
\]

To complete the proof we recall that $\pi_F^1$ is a faithful representation of $\mathcal O_n^{(0)}\otimes\mathcal O_{m}^{(0)}$, so its restriction to $\mathbb K\otimes\mathcal O_m^{(0)}$ is also faithful, implying the injectivity of $\Delta_q$.
 \end{proof}

\begin{remark}
Note that any $\mathbb T$-action on the $C^*$-algebra of compact operators is inner. Hence,  Corollary 5.16 and Example 5.17 in \cite{mey_wor} imply that any twisted tensor product with the compact operators is isomorphic to the usual tensor product. This gives a short proof of the above proposition, though without an explicit formula for the isomorphism.
\end{remark}

\begin{remark}\label{rem_iq_1}
Evidently, $\mathcal I_q$ is a closed linear span of the family
\[
\{\, t_{\mu_2}(1-P)t_{\nu_2}^*s_{\mu_1}(\mathbf 1 -Q)s_{\nu_1}^*,\
\mu_1,\nu_1\in\Lambda_n,\ \mu_2,\nu_2\in\Lambda_m\}\subset \mathcal I_1^q.
\]
Moreover, $c.l.s.\{t_{\mu_2}(1-P)t_{\nu_2}^*,\  \mu_2,\nu_2\in\Lambda_m\}=\mathbb K\subset\mathcal O_m^{(0)}$. It is easy to see that
\[
\alpha_{\mu}(t_{\mu_2}(1-P)t_{\nu_2}^*)=q^{|\mu|(|\mu_2|-|\nu_2|)}
t_{\mu_2}(1-P)t_{\nu_2}^*,
\]
so every $\alpha_\mu\in \alpha^q$ can be regarded as an element of $\mathsf{Aut} (\mathbb K)$.
\end{remark}

A moment reflection and Proposition \ref{str_I1q} give the following corollary
\begin{proposition}\label{str_Iq_1}
Restriction of $\Delta_{q,1}$ to $\mathbb{K}\otimes_{\alpha^q}\mathbb K \subset \mathcal{O}_m^{(0)} \otimes_{\alpha^q} \mathbb{K}$ gives an isomorphism \[ \Delta_{q,1} \colon \mathbb{K}\otimes_{\alpha^q}\mathbb K \rightarrow \mathcal I_q . \]
\end{proposition}

To deal with $\mathcal I_2^q$, we consider the family $\beta^q=\{\beta_{\mu},\ \mu\in\Lambda_m\}\subset \mathsf{Aut}(\mathcal O_n^{(0)})$ defined as
\[
\beta_\mu (s_j)=q^{-|\mu|}s_j,\ \beta_{\mu}(s_j^*)=q^{|\mu|}s_j^*,\ j=\overline{1,n}.
\]
\begin{proposition}\label{str_I2q}
One has an isomorphism $\Delta_{q,2}\colon\mathcal{O}_n^{(0)}\otimes_{\beta^q}\mathbb K\rightarrow \mathcal I_2^q$.
\end{proposition}

Obviously, $\Delta_{q,2}$ induces the isomorphism
$\mathbb K\otimes_{\beta^q}\mathbb K\simeq\mathcal I_q$, where the first term is an ideal in $\mathcal O_n^{(0)}$ and the second in $\mathcal O_m^{(0)}$ respectively.

Write
\[
\varepsilon_n\colon\mathbb K\rightarrow \mathcal O_n^{(0)}, \quad \varepsilon_m\colon\mathbb K\rightarrow \mathcal O_m^{(0)},
\]
for the canonical embeddings and
\[
q_n\colon\mathcal O_n^{(0)}\rightarrow\mathcal O_n, \quad q_m\colon\mathcal O_m^{(0)}\rightarrow \mathcal{O}_m,
\]
for the quotient maps. Let also
\[
\varepsilon_{q,j}\colon\mathcal I_q\rightarrow\mathcal I_{j}^q, \quad j=1,2,
\]
be the embeddings and
\[
\pi_{q,j}\colon\mathcal I_{j}^q\rightarrow \mathcal I_j^q/\mathcal I_q, \quad j=1,2,
\]
the quotient maps. Notice also that the families $\alpha^q\subset\mathsf{Aut}(\mathcal O_m^{(0)})$,
$\beta^q\subset\mathsf{Aut}(\mathcal O_n^{(0)})$ determine families of automorphisms of $\mathcal O_m$ and $\mathcal O_n$ respectively, also denoted by $\alpha^q$ and $\beta^q$.

\begin{theorem}\label{i1i2q_comm}
One has the following isomorphism of extensions
\[
\begin{tikzcd}
 0 \arrow[r] & \mathcal{I}_q \arrow[r, "\varepsilon_{q,1}"] \arrow[d,"\Delta_{\alpha^q}\circ\Delta_{q,1}^{-1}"] & \mathcal{I}_1^q \arrow[d, "\Delta_{\alpha^q}\circ\Delta_{q,1}^{-1}"] \arrow[r, "\pi_{q,1}"] & \mathcal{I}_1^q / \mathcal{I}_q \arrow[d, "\simeq"] \arrow[r] & 0 \\
  0 \arrow[r] & \mathbb{K} \otimes \mathbb{K} \arrow[r,  "\varepsilon_m\otimes\id_\mathbb{K}"] &\mathcal{O}_m^0 \otimes \mathbb{K}  \arrow[r, "q_m\otimes\id_{\mathbb K}"] & \mathcal{O}_m\otimes \mathbb{K}  \arrow[r] & 0
\end{tikzcd}
\]
and
\[
\begin{tikzcd}
 0 \arrow[r] & \mathcal{I}_q \arrow[r, "\varepsilon_{q,2}"] \arrow[d,"\Delta_{\beta^q}\circ\Delta_{q,2}^{-1}"] & \mathcal{I}_2^q \arrow[d, "\Delta_{\beta^q}\circ\Delta_{q,2}^{-1}"] \arrow[r,"\pi_{q,2}"] & \mathcal{I}_2^q / \mathcal{I}_q \arrow[d, "\simeq"] \arrow[r] & 0 \\
  0 \arrow[r] & \mathbb{K} \otimes \mathbb{K} \arrow[r, "\varepsilon_n\otimes
  \id_\mathbb{K}"] & \mathcal{O}_n^0\otimes\mathbb{K} \arrow[r, "q_n\otimes\id_{\mathbb K} "] & \mathcal{O}_n\otimes\mathbb{K} \arrow[r] & 0
\end{tikzcd}
\]
\end{theorem}

\begin{proof}
Indeed, each  row in diagram (\ref{comm_diag_tw}) below is exact and every non-dashed vertical arrow  is an isomorphism. The bottom left and bottom right squares are commutative due to (\ref{ktwist_func}). The top left square is commutative due to the arguments in the proof of Proposition \ref{str_I1q} combined with Remark \ref{rem_iq_1}. Hence there exists a unique isomomorphism
 \[
 \Phi_{q,1}\colon\mathcal I_1^q/\mathcal I_q\rightarrow\mathcal O_m\otimes_{\alpha^q}\mathbb K,
 \]
 making the diagram (\ref{comm_diag_tw}) commutative
\begin{equation}\label{comm_diag_tw}
\begin{tikzcd}
 0 \arrow[r] & \mathcal{I}_q \arrow[r, "\varepsilon_{q,1}"] \arrow[d, "\Delta_{q,1}^{-1}"] & \mathcal{I}_1^q \arrow[d, "\Delta_{q,1}^{-1}"] \arrow[r, "\pi_{q,1}"] & \mathcal{I}_1^q / \mathcal{I}_q \arrow[d, dashrightarrow, "\Phi_{q,1}"] \arrow[r] & 0 \\
 0 \arrow[r] & \mathbb{K} \otimes_{\alpha^q} \mathbb{K} \arrow[r, "\varepsilon_{m}\otimes_{\alpha^q}^{\alpha^q}"] \arrow[d, "\Delta_{\alpha^q}"] & \mathcal{O}_m^0 \otimes_{\alpha^q}\mathbb{K}  \arrow[d, "\Delta_{\alpha^q}"] \arrow[r, " q_m \otimes_{\alpha_q}^{\alpha^q}"] & \mathcal{O}_m\otimes_{\alpha^q}\mathbb{K}   \arrow[d, "\Delta_{\alpha^q}"] \arrow[r] & 0 \\
  0 \arrow[r] & \mathbb{K} \otimes \mathbb{K} \arrow[r, "\varepsilon_m\otimes\id_\mathbb{K} "] & \mathcal{O}_m^0\otimes\mathbb{K}  \arrow[r, "q_m\otimes\id_{\mathbb K} "] & \mathcal{O}_m\otimes\mathbb{K} \arrow[r] & 0 \\
\end{tikzcd}
\end{equation}
The proof for $\mathcal I_2^q$ is similar.
 \end{proof}

The following Lemma follows from the fact that $\mathcal M_q = \mathcal{I}_1^q + \mathcal{I}_2^q$.

\begin{lemma}\label{lemextmq}
\[
\mathcal{M}_q / \mathcal{I}_q \simeq  \mathcal{I}_1^q/\mathcal{I}_q \oplus \mathcal{I}_2^q/\mathcal{I}_q \simeq \mathcal{O}_m \otimes \mathbb{K}\oplus\mathcal O_n\otimes\mathbb{K}.
\]
\end{lemma}

Theorem \ref{i1i2q_comm} implies that $\mathcal I_q, \mathcal I^q_1, \mathcal I^q_2$ are stable $C^*$-algebras. It follows from \cite{Rordam}, Proposition 6.12, that an extension of a stable $C^*$-algebra by $\mathbb{K}$ is also stable. Thus, Lemma \ref{lemextmq} implies immediately the following important corollary.

\begin{corollary}
For any $q\in\mathbb{C}$, $|q|=1$, the $C^*$-algebra $\mathcal{M}_q$ is stable.
\end{corollary}

Denote the Calkin algebra by $Q$. Recall that for $C^*$-algebras $\mathcal A$ and $\mathcal B$ the isomorphism
\[
\mathsf{Ext}( \mathcal A \oplus\mathcal  B,\mathbb{K}) \simeq \mathsf{Ext}(\mathcal A,\mathbb{K}) \oplus \mathsf{Ext}(\mathcal B,\mathbb{K})
\]
is given as follows. Let
\[
\iota_1 : \mathcal A \rightarrow \mathcal A \oplus\mathcal B,\quad  \iota_1(a) = (a,0),\quad
\iota_2 :\mathcal B \rightarrow \mathcal A \oplus \mathcal B,\quad \iota_2(b) = (0,b).
\]
For a Busby invariant $\tau : \mathcal A \oplus\mathcal B \rightarrow Q$ define
\[\mathsf{F} : \mathsf{Ext}(\mathcal A \oplus \mathcal B, \mathbb{K}) \rightarrow \mathsf{Ext}(\mathcal A, \mathbb{K}) \oplus \mathsf{Ext}(\mathcal B, \mathbb{K}),\quad
\mathsf{F}(\tau) = (\tau \circ \iota_1, \tau \circ \iota_2).
\]
It can be shown, see \cite{higson}, that $\mathsf{F}$ determines a group isomorphism.

\begin{remark}\label{rem_beta}
Consider an extension
\begin{equation}\label{ext_busby}
\begin{tikzcd}
  0 \arrow[r] & \mathcal B \arrow[r] & \mathcal E \arrow[r] & \mathcal A \arrow[r]  & 0
\end{tikzcd}
\end{equation}
Let $i\colon\mathcal B \to M(\mathcal B)$ be the canonical embedding. Define $\beta\colon\mathcal E\rightarrow M(\mathcal B)$, to be the unique map such that
\[
\beta(e)i(b) = i(eb),\quad\mbox{for every}\ b \in \mathcal B,\ 
e\in\mathcal E.
\]
Then the Busby invariant $\tau$ is the unique map which makes the diagram commute.
\[
\begin{tikzcd}
 0 \arrow[r] & \mathcal B \arrow[r, "i"] & M(\mathcal B) \arrow[r] & M(\mathcal B)/\mathcal B \arrow[r] & 0 \\
 0 \arrow[r] & \mathcal B \arrow[r]\arrow[equal]{u} & \mathcal E \arrow[r]  \arrow[u, "\beta"]& \mathcal A \arrow[r] \arrow[u, "\tau"] & 0
\end{tikzcd}
\]
We will use both notations $[\mathcal E]$ and $[\tau]$ in order to denote the class of the extension (\ref{ext_busby}) in 
$\mathsf{Ext}(\mathcal A,\mathcal B)$.
\end{remark}

Let $[\mathcal{M}_q] \in \mathsf{Ext}( \mathcal{I}_1^q/\mathcal{I}_q \oplus \mathcal{I}_2^q/\mathcal{I}_q, \mathcal I_q)$, $[\mathcal I_1^q] \in \mathsf{Ext}(\mathcal I_1^q / \mathcal{I}_q, \mathcal{I}_q)$, $[\mathcal I_2^q] \in \mathsf{Ext}(\mathcal I_2^q / \mathcal{I}_q, \mathcal{I}_q)$ respectively be the classes of the following extensions
\begin{gather*}
0 \rightarrow \mathcal I_q \rightarrow \mathcal M_q \rightarrow \mathcal{I}_1^q/\mathcal{I}_q \oplus \mathcal{I}_2^q/\mathcal{I}_q \rightarrow 0,
\\
0 \rightarrow \mathcal I_q \rightarrow \mathcal I^q_1 \rightarrow \mathcal I^q_1 / \mathcal I_q \rightarrow 0,
\\
0 \rightarrow \mathcal I_q \rightarrow \mathcal I^q_2 \rightarrow \mathcal I^q_2 / \mathcal I_q \rightarrow 0.
\end{gather*}

\begin{lemma}\label{lem_qdecomp}
\[
[\mathcal{M}_q] = ([\mathcal{I}_1^q], [\mathcal{I}_2^q]) \in \mathsf{Ext}(\mathcal I_1^q / \mathcal{I}_q, \mathcal{I}_q) \oplus \mathsf{Ext}(\mathcal I_2^q / \mathcal{I}_q, \mathcal{I}_q) \simeq \mathsf{Ext}( \mathcal{I}_1^q/\mathcal{I}_q \oplus \mathcal{I}_2^q/\mathcal{I}_q,\mathcal{I}_q).
\]
\end{lemma}

\begin{proof}
Consider the following morphism of extensions:
\[
\begin{tikzcd}[row sep=scriptsize, column sep=scriptsize]
& \mathcal{I}_q \arrow[dl, equal] \arrow[rr, "i"] \arrow[dd, equal] & & M(\mathcal{I}_q) \arrow[dd, equal] \arrow[rr] & & M(\mathcal{I}_q)/\mathcal{I}_q \arrow[dd, equal] \\
\mathcal{I}_q \arrow[rr, crossing over] \arrow[dd, equal] & & \mathcal{I}_1^q \arrow[ur, "\beta_1"] \arrow[rr,  crossing over] & & \mathcal{I}_1^q/\mathcal{I}_q  \arrow[ur, "\tau_{\mathcal{I}_1^q}"] \\
& \mathcal{I}_q \arrow[dl, equal] \arrow[rr] & & M(\mathcal{I}_q)  \arrow[rr] & & M(\mathcal{I}_q) / \mathcal{I}_q \arrow[from=dl, "\tau_{\mathcal{M}_q}"] \\
\mathcal{I}_q \arrow[rr] & & \mathcal{M}_q \arrow[ur, "\beta_2"] \arrow[from=uu, crossing over, hookrightarrow] \arrow[rr] & & \mathcal{I}_1^q/\mathcal{I}_q \oplus \mathcal{I}_2^q/\mathcal{I}_q \arrow[from=uu, crossing over] \\
\end{tikzcd}
\]
Here
\[
\beta_1\colon\mathcal I_1^q \rightarrow M(\mathcal I_q),\quad \beta_2\colon\mathcal M_q \rightarrow M(\mathcal I_q),
\]
are homomorphisms introduced in Remark \ref{rem_beta}, the vertical arrow
\[
j_1\colon \mathcal{I}_1^q \hookrightarrow \mathcal{M}_q
\]
is the inclusion, and the vertical arrow
\[
\iota_1 : \mathcal{I}_1^q / \mathcal{I}_q \rightarrow \mathcal{I}_1^q / \mathcal{I}_q \oplus \mathcal{I}_2^q / \mathcal{I}_q
\]
has the form $\iota_1(x) = (x, 0)$.

Notice that for every $b \in  \mathcal{I}_q$ and
$x \in \mathcal{I}_1^q$ one has
\[
(\beta_2 \circ j_1)(x) i(b) = i(j_1(x)b) = i(xb)=\beta_1(x)i(b).
\]
By the uniqueness of $\beta_1$, we get $\beta_2 \circ j_1 = \beta_1$. Thus the following diagram commutes
\[
\begin{tikzcd}
\mathcal I_1^q \arrow[r, "\beta_1"] \arrow[d, hookrightarrow] & M(\mathcal I_q) \arrow[d, equal] \\
 \mathcal M_q \arrow[r, "\beta_2"] & M(\mathcal I_q)
\end{tikzcd}
\]
Further, Remark \ref{rem_beta} implies that for  Busby invariants $\tau_{\mathcal I_1^q}$ and $\tau_{\mathcal M_q}$  the squares below are commutative
\[
\begin{tikzcd}
 M(\mathcal I_q) \arrow[r] & M(\mathcal I_q)/\mathcal I_q \\
 \mathcal I_1^q \arrow[r]  \arrow[u, "\beta_1"]& \mathcal I_1^q/\mathcal I_q \arrow[u, "\tau_{\mathcal I_1^q}"]
\end{tikzcd}\quad
\begin{tikzcd}
 M(\mathcal I_q) \arrow[r] & M(\mathcal I_q)/\mathcal I_q \\
 \mathcal M_q \arrow[r]  \arrow[u, "\beta_2"]& \mathcal I_1^q/\mathcal I_q
 \oplus \mathcal I_2^q/\mathcal I_q\arrow[u, "\tau_{\mathcal M_q}"]
\end{tikzcd}
\]
Hence the square
\[
\begin{tikzcd}
 \mathcal I_1^q/\mathcal I_q \arrow[r,"\tau_{\mathcal I_1^q}"] \arrow[d,"\iota_1"]& M(\mathcal I_q)/\mathcal I_q \arrow[d,equal]\\
 \mathcal I_1^q/\mathcal I_q\oplus\mathcal I_2^q/\mathcal I_q \arrow[r, "\tau_{\mathcal M_q}"] & M(\mathcal I_q)/\mathcal I_q
\end{tikzcd}
\]
is also commutative. Thus,
$\tau_{\mathcal{I}_1^q} = \tau_{\mathcal{M}_q} \circ \iota_1$. By the same arguments we get $\tau_{\mathcal I_2^q}=\tau_{\mathcal M_q}\circ\iota_2$, where
\[
\iota_2\colon\mathcal I_2^q/\mathcal I_q\rightarrow \mathcal{I}_1^q / \mathcal{I}_q \oplus \mathcal{I}_2^q / \mathcal{I}_q,\quad \iota_2 (y)=(0,y).
\]
Thus
\[ [\tau_{\mathcal{M}_q}] = ([\tau_{\mathcal{M}_q}\circ\iota_1], [\tau_{\mathcal{M}_q} \circ\iota_2]) = ([\tau_{\mathcal{I}_1^q}], [\tau_{\mathcal{I}_2^q}]).
\]
 \end{proof}

In the following theorem we give a description of all ideals in $\mathcal E_{n,m}^q$.

\begin{theorem}\label{ideals_enmq}
Any ideal $J\subset\mathcal E_{n,m}^q$ coincides with one of $\mathcal I_q$,
$\mathcal I_1^q$, $\mathcal I_2^q$, $\mathcal M_q$.
\end{theorem}

\begin{proof}
First we notice that $\mathcal I_1^q/\mathcal I_q\simeq \mathcal O_m\otimes\mathbb K$, $\mathcal I_2^q/\mathcal I_q\simeq \mathcal O_n\otimes\mathbb K$ are simple. Hence for any ideal $\mathcal J$ such that $\mathcal I_q\subseteq\mathcal J\subseteq\mathcal I_1^q$ or  $\mathcal I_q\subseteq\mathcal J\subseteq\mathcal I_2^q$,
one has  $\mathcal J=\mathcal I_q$, or $\mathcal J=\mathcal I_1^q$, or $\mathcal J=\mathcal I_2^q$.

Further, using the fact that $\mathcal M_q=\mathcal I_1^q+\mathcal I_2^q$ and $\mathcal I_q=\mathcal I_1^q\cap\mathcal I_2^q$ we get
\[
\mathcal M_q/\mathcal I_1^q\simeq\mathcal I_2^q/\mathcal I_q\simeq\mathcal O_n\otimes\mathbb K.
\]
So if $\mathcal I_1^q \subseteq \mathcal J\subseteq\mathcal M_q$, then again either
$\mathcal J=\mathcal I_1^q$ or $\mathcal J=\mathcal M_q$. Obviously, the same result holds for $\mathcal I_2^q \subseteq \mathcal J\subseteq\mathcal M_q$.

Below, see Proposition \ref{OmnIsPurelyInfinite}, we show that $\mathcal E_{n,m}^q/\mathcal M_q$ is simple and purely infinite. In particular, $\mathcal M_q$ contains any ideal in $\mathcal E_{n,m}^q$, see Corollary \ref{mq_unique}.

Let $\mathcal J\subset \mathcal E_{n,m}^q$ be an ideal and $\pi$ be a representation of $\mathcal{E}_{n,m}^q$ such that $\ker \pi = \mathcal{J}$. Notice that the Fock component $\pi_1$ in the Wold decomposition of $\pi$ is zero. Thus, by Theorem \ref{wold_dec},
\begin{equation}\label{pi_decomp}
\pi=\pi_2\oplus\pi_3\oplus\pi_4,
\end{equation}
and $\mathcal J=\ker\pi=\ker\pi_2\cap\ker\pi_3\cap\ker\pi_4$. Let us describe these kernels. Suppose that the component $\pi_2$ is non-zero. Since $\pi_2(\mathbf 1-Q)=0$ and $\pi_2(\mathbf 1-P)\ne 0$, we have
\[
\mathcal I_1^q\subseteq\ker\pi_2\subsetneq\mathcal M_q,
\]
implying $\ker\pi_2=\mathcal I_1^q$. Using the same arguments, one can deduce that if the component $\pi_3$ is non-zero, then $\ker\pi_3=\mathcal I_2^q$, and if $\pi_4$ is non-zero, then $\ker\pi_4=\mathcal M_q$.

Finally, if in (\ref{pi_decomp}) $\pi_2$ and $\pi_3$ are non-zero then $\mathcal J=\ker\pi=\mathcal I_q$. If either $\pi_2 \neq 0$ and $\pi_3 = 0$ or $\pi_3\neq 0$ and $\pi_2=0$, then either  $\mathcal J=\mathcal I_1^q$ or $\mathcal J=\mathcal I_2^q$. In the case $\pi_2=0$ and $\pi_3=0$ one has $\mathcal J=\ker\pi_4=\mathcal M_q$.
 \end{proof}

\begin{corollary}\label{essential}
 All ideals in $\mathcal E_{n,m}^q$ are essential. The ideal $\mathcal I^q$ is the unique minimal ideal.
\end{corollary}

In particular, the extension
\[ 0
\rightarrow \mathcal{I}_q \rightarrow \mathcal{M}_q \rightarrow \mathcal{I}_1^q/\mathcal{I}_q \oplus \mathcal{I}_2^q/\mathcal{I}_q \rightarrow 0
\]
is essential.
Indeed, the ideal $\mathbb K=\mathcal{I}_q\subset \mathcal E_{n,m}^q$ is the unique minimal ideal. Since an ideal of an ideal in a $C^*$-algebra is an ideal in the whole algebra, $\mathcal{I}_q$ is the unique minimal ideal in $\mathcal M_q$, thus it is essential in $\mathcal M_q$.

The following proposition is a corollary of Voiculescu's Theorem, see Theorem 15.12.3 of \cite{Black}.

\begin{proposition}\label{Voiculescu}
Let $\mathcal E_1, \mathcal E_2$ be two essential extensions of a nuclear $C^*$-algebra $\mathcal A$ by $\mathbb{K}$. If $[\mathcal E_1] = [\mathcal E_2] \in \mathsf{Ext}(\mathcal A, \mathbb{K})$ then $\mathcal E_1 \simeq \mathcal E_2$.
\end{proposition}

\begin{theorem}
For any $q\in\mathbb C$, $|q|=1$, one has
$\mathcal{M}_q \simeq \mathcal{M}_1$.
\end{theorem}

\begin{proof}
By Theorem \ref{i1i2q_comm}, $[\mathcal{I}_1^q]\in\mathsf{Ext}(\mathcal{O}_m \otimes \mathbb{K},\mathbb{K})$, and  $[\mathcal{I}_2^q] \in \mathsf{Ext}(\mathcal{O}_n \otimes \mathbb{K}, \mathbb{K})$ do not depend on $q$. By Lemma \ref{lem_qdecomp}, $[\mathcal{M}_q]$ does not depend on $q$. Thus by Corollary \ref{essential} and Proposition \ref{Voiculescu}, $\mathcal{M}_q \simeq \mathcal{M}_1$.
 \end{proof}

\section{The multiparameter case}

We now turn to the $C^*$-algebra $\mathcal O_n\otimes_q\mathcal O_m:=\mathcal E_{n,m}^q/\mathcal M_q$. Our goal is to show that it is isomorphic to $\mathcal O_n\otimes\mathcal O_m$. More generally, we show that even multiparameter twists of $\mathcal O_n\otimes\mathcal O_m$ are impossible. This is very specific for the Cuntz algebra. Indeed, recall that the tensor product of $C(S^1)\otimes C(S^1)$ \emph{does} admit twists, the famous rotation algebra $A_q$ being the result. Also twists of tensor products of Toeplitz algebras have been considered \cite{weber, jeu_pinto}. However, the tensor product of Cuntz algebras may not be twisted as we will see.

Our proof  uses deep theory from the classification of $C^*$-algebras: We  use Kirchberg's seminal result from the 1990's \cite{Kirchberg}, stating that two unital, separable, nuclear, simple, purely infinite $C^*$-algebras are isomorphic if and only if they have the same $K$-groups.

Note that twists of $C^*$-algebras $A$ with Cuntz algebras $\mathcal O_m$ were first considered by Cuntz in 1981 \cite{Cuntz81}, see also the PhD thesis of Neub\"user from 2000 \cite{neubueser}. We  reformulate Cuntz's definition in terms of universal $C^*$-algebras later. Let us now focus on the case $A=\mathcal O_n$.

\begin{definition}\label{DefOmn}
 Let $2\leq n,m < \infty$ and let $\Theta=(q_{ij})$ be a matrix with $q_{ij}\in S^1$ being scalars of absolute value one, for $i=1,\ldots,n$ and $j=1,\ldots,m$. We define the twist of two Cuntz algebras $\mathcal O_n$ and $\mathcal O_m$ as the universal $C^*$-algebra $\mathcal O_n\otimes_\Theta \mathcal O_m$ generated by isometries $s_1,\ldots,s_n$ and $t_1,\ldots,t_m$ with $\sum^n_{i=1}s_is_i^*=\sum^m_{j=1}t_jt_j^*=1$ and $s_it_j=q_{ij}t_js_i$ for all $i$ and $j$. 
\end{definition}

Observe that for $q_{ij}=1$ for all $i$ and $j$, we obtain $\mathcal O_n\otimes_\Theta \mathcal O_m=\mathcal O_n\otimes \mathcal O_m$, while the choice $q_{ij}=\bar q$ yields the one parameter deformation $\mathcal O_n\otimes_q\mathcal O_m:=\mathcal E_{n,m}^q/\mathcal M_q$, which can be seen as follows.

We have $s_it_j=q_{ij}t_js_i$ for all $i$ and $j$ if and only if $s_i^*t_j=\bar q_{ij}t_js_i^*$ for all $i$ and $j$. Indeed, assuming the former relations and using $t_k^*t_i=\delta_{ik}$, we infer
\[s_i^*t_j=\sum_k t_kt_k^*s_i^*t_j=\sum_k\bar q_{ik} t_ks_i^*t_k^*t_j=\bar q_{ij} t_js_i^*.\]
Conversely, fixing $i$ and $j$ and assuming $s_it_j^*=\bar q_{ij}t_j^*s_i$, we obtain
\[(s_it_j-q_{ij}t_js_i)^*(s_it_j-q_{ij}t_js_i)=1-\bar q_{ij}s_i^*t_j^*s_it_j-q_{ij}t_j^*s_i^*t_js_i+1=0.\]
Moreover, note that we may view $\mathcal O_n\otimes_\Theta \mathcal O_m$ as a twisted tensor product in the following sense. Putting Cuntz's definition from \cite{Cuntz81} to a language of universal $C^*$-algebras, we may say that given a unital $C^*$-algebra $A$ and automorphisms $\alpha_1,\ldots,\alpha_m$ of $A$ satisfying $\alpha_i\circ\alpha_j=\alpha_j\circ\alpha_i$, we may define a twisted tensor product $A\times_{(\alpha_1,\ldots,\alpha_m)}\mathcal O_m$ as the universal unital $C^*$-algebra generated by all elements $a\in A$ (including the relations of $A$) and isometries $t_1,\ldots, t_m$ such that $\sum_kt_kt_k^*=1$ and $t_j a=\alpha_j(a) t_j$, for all $j=1,\ldots, m$ and $a\in A$. We observe that  $\mathcal O_n\otimes_\Theta \mathcal O_m$ is such a twisted tensor product, if we put $\alpha_j(s_i):=\bar q_{ij}s_i$ for $A=\mathcal O_n=C^*(s_1,\ldots,s_n)$.

Let us quickly show that $\mathcal O_n\otimes_\Theta \mathcal O_m$ may be represented concretely on a Hilbert space. 
Firstly, represent $\mathcal O_m$  on $\ell^2(\mathbb N_0)$ by $\hat T_j e_k:=e_{\beta_j(k)}$, where  
\[\beta_1,\ldots,\beta_m:\N_0\to\N_0\] 
are injective functions with disjoint  ranges  such that the union of their ranges is all of $\N_0$. Secondly, choose any representation $\sigma:\mathcal O_n\to L(\ell^2(\mathbb N_0))$ mapping  $s_i\mapsto\hat S_i$. Thirdly, let $\tilde S$ be the unilateral shift on $\ell^2(\mathbb Z)$. Fourthly, define the diagonal unitaries $U_i$ on $\ell^2(\mathbb N_0\times \mathbb Z)$ by $U_ie_{m,n}:=\zeta_i(m,n)e_{m,n}$, where the scalars $\zeta_i(m,n)$ with $\lvert\zeta_i(m,n)\rvert=1$ obey the inductive rule $\zeta_i(\beta_j(m),n+1)=q_{ij}\zeta_i(m,n)$. Finally,  $\pi:\mathcal O_n\otimes_\Theta \mathcal O_m\to B(\ell^2(\mathbb N_0\times\mathbb N_0
\times\mathbb Z))$ with
\[\pi(s_i):=\hat S_i\otimes U_i, \qquad\pi(t_j):=\id\otimes \hat T_j\otimes \tilde S\]
exists by the universal property.

\subsection{$\mathcal O_n\otimes_\Theta \mathcal O_m$ is nuclear}

Let us now begin with collecting the ingredients for an application of Kirchberg's Theorem. The first step is to verify that $\mathcal O_n\otimes_\Theta \mathcal O_m$ is nuclear. 
For the one parameter deformation $\mathcal O_n\otimes_q\mathcal O_m$, this is a consequence of Corollary \ref{EIsNuclear}. For the multiparameter deformations, we
 use  the following lemma by Rosenberg, which is actually a statement about crossed products with the semigroup $\mathbb N_0$ by the endomorphism $b\mapsto sbs^*$ of $B$. Recall that the bootstrap class $\mathcal N$ (also called UCT class) is a class of nuclear $C^*$-algebras which is closed under many operations; see \cite[ch. IX, sect. 22.3]{Black}  for more on $\mathcal N$.

\begin{lemma}[{\cite[th. 3]{Rosenberg}}]\label{LemRosenberg}
 Let $A$ be a unital $C^*$-algebra, and let $B\subset A$ be a nuclear $C^*$-subalgebra containing the unit of $A$. Let $s\in A$ be an isometry such that $sBs^*\subset B$ and $A=C^*(B,s)$, i.e. let $A$ be generated by $B$ and $s$. Then $A$ is nuclear. Moreover, if $B$ is in the bootstrap class $\mathcal N$, so is $A$.
\end{lemma} 

Let us denote by $s_{\mu}$ the product $s_{\mu_1}\ldots s_{\mu_k}$, where $\mu=(\mu_1,\ldots,\mu_k)$ is a multi index of length $|\mu|=k$ and $\mu_1,\ldots,\mu_k\in\{1,\ldots,n\}$.
Given $k\in \N_0$, denote by $\mathcal G_k$ the $C^*$-subalgebra of $\mathcal O_n\otimes_\Theta \mathcal O_m$ defined as
\[\mathcal G_k:=C^*(s_{\mu}xs^*_{\nu}\mid x\in\mathcal O_m, |\mu|=|\nu|=k)\subset\mathcal O_n\otimes_\Theta \mathcal O_m.\]
Here, we denote by $\mathcal O_m$ the $C^*$-subalgebra $\mathcal G_0=C^*(t_1,\ldots,t_m)\subset\mathcal O_n\otimes_\Theta \mathcal O_m$ (which is isomorphic to $\mathcal O_m$, of course).
The closed union of all $\mathcal G_k$ is denoted by $B_0$, so
\[B_0:=C^*(s_{\mu}xs^*_{\nu}\mid x\in\mathcal O_m, |\mu|=|\nu|\in\mathbb N_0)\subset\mathcal O_n\otimes_\Theta \mathcal O_m.\]

\begin{proposition}\label{OmnIsNuclear}
We have the following:
\begin{itemize}
\item[(a)]  $\mathcal G_k\subset\mathcal O_n\otimes_\Theta \mathcal O_m$ is nuclear and in $\mathcal N$ for all $k\in\N_0$.
\item[(b)] $B_0\subset\mathcal O_n\otimes_\Theta \mathcal O_m$ is nuclear and in $\mathcal N$.
\item[(c)] $\mathcal O_n\otimes_\Theta \mathcal O_m$ is nuclear and  in $\mathcal N$.
\end{itemize}
\end{proposition}
\begin{proof}
For (a), we show that $\mathcal G_k$ is isomorphic to $M_{n^k}(\C)\otimes\mathcal O_m$, where $M_{n^k}(\C)$ denotes the algebra of $n^k\times n^k$ matrices with complex entries. For doing so, let $k>0$ and write $M_{n^k}(\C)\otimes\mathcal O_m$ as the universal $C^*$-algebra generated by elements $e_{\mu\nu}$ for multi-indices $\mu, \nu$ in $\{1,\ldots,n\}$ of length $k$, together with  isometries $t_1,\ldots,t_m$ 
such that $\sum_k t_kt_k^*=1$
and  $e_{\mu\nu}e_{\rho\sigma}=\delta_{\nu\rho}e_{\mu\sigma}$, $e_{\mu\nu}^*=e_{\nu\mu}$ and $e_{\mu\nu}t_i=t_ie_{\mu\nu}$.

It is then easy to see that the  elements $\hat e_{\mu\nu}:=s_{\mu}s^*_{\nu}$ and $\hat t_i:= \sum_{|\rho|=k}s_{\rho}t_is^*_{\rho}$ in $\mathcal G_k$ fulfill the relations of $M_{n^k}(\C)\otimes\mathcal O_m$. As this $C^*$-algebra is simple, $\mathcal G_k$ is isomorphic to $M_{n^k}(\C)\otimes\mathcal O_m$. Now, $\mathcal O_m$ is in $\mathcal N$ and the bootstrap class is closed under tensoring with matrix algebras, so $M_{n^k}(\C)\otimes\mathcal O_m$ is in $\mathcal N$.

As for (b), note that we have $\mathcal G_k\subset \mathcal G_{k+1}$ for all $k$. Indeed, check that $s_{\mu}t_is^*_{\nu}=\sum_k\bar q_{ki}s_{\mu}s_kt_is^*_ks^*_{\nu}\in\mathcal G_{k+1}$. We may then view  $B_0$ as the inductive limit of the system 
$\mathcal G_1 \subset \mathcal G_2 \subset \ldots$. 
As $\mathcal N$ is closed under inductive limits, we get the result by (a). 

Finally, for (c),  put $B_{j+1}:=C^*(B_j, s_{j+1})\subset\mathcal O_n\otimes_\Theta \mathcal O_m$ for $j=0,\ldots,n-1$. We  prove inductively that $B_{j+1}$ is in $\mathcal N$ and that $B_{j+1}=C^*(B_0,s_1,\ldots,s_{j+1})$, using Lemma \ref{LemRosenberg}. By induction hypothesis (or by (b) in the case $j=0$), the $C^*$-algebra $B_j$ is nuclear and in $\mathcal N$, it is unital with the same unit as $B_{j+1}$ and it is contained in $B_{j+1}$. Therefore, to apply Rosenberg's lemma, we only have to check, that $s_{j+1}B_js^*_{j+1}\subset B_j=C^*(B_0,s_1,\ldots,s_j)$. 

For this, let $s_{\mu}xs^*_{\nu}$ be in $B_0$ for some $x\in\mathcal O_m$ and $|\mu|=|\nu|\in\N_0$. Then $s_{j+1}s_{\mu}xs^*_{\nu}s^*_{j+1}\in B_0\subset B_j$. We conclude $s_{j+1}B_0s^*_{j+1}\subset B_j$. For $k=1,\ldots,j$, we have $s_{(j+1,k)}s^*_{(k,j+1)}\in B_0$, where $(j+1,k)$ and $(k,j+1)$ denote multi-indices of length two. Thus:
\[s_{j+1}s_ks^*_{j+1}=s_{j+1}s_ks^*_{j+1}s^*_ks_k=(s_{(j+1,k)}s^*_{(k,j+1)})s_k\in B_j\]
Thus, if $x=x_1\ldots x_k$ is a product of elements $x_l\in B_0\cup\{s_1,s_1^*,\ldots, s_j, s_j^*\}$, for $l=1,\ldots,k$, we have
\[s_{j+1}xs_{j+1}^*=(s_{j+1}x_1s_{j+1}^*)\ldots(s_{j+1}x_ks_{j+1}^*)\in B_j.\]
This proves $s_{j+1}C^*(B_0,s_1,\ldots,s_j)s^*_{j+1}\subset B_j$ and we are done. 

Since $\mathcal O_n\otimes_\Theta \mathcal O_m=B_n$, we obtain the stated result.
\end{proof}

\subsection{$\mathcal O_n\otimes_\Theta \mathcal O_m$ is purely infinite}

Next, we are going to show that $\mathcal O_n\otimes_\Theta \mathcal O_m$ is purely infinite. Recall, that a unital $C^*$-algebra $A$ is called purely infinite, if for all nonzero elements $x\in A$, there are $a,b\in A$ such that $axb=1$. In particular, $A$ is then simple.

In our proof,  we  follow Cuntz's work on $\mathcal O_n$ from 1977 \cite{cun} and we  adapt it to $\mathcal O_n\otimes_\Theta \mathcal O_m$.
The idea is to find a faithful expectation $\phi$ on $\mathcal O_n\otimes_\Theta \mathcal O_m$, such that an element $0\neq x\in\mathcal O_n\otimes_\Theta \mathcal O_m$ (or rather an element $y$ close to $x^*x$) is mapped to a self-adjoint complex valued matrix with controlled norm. By linear algebra, we then can find a minimal projection $e$, which projects onto a one-dimensional subspace corresponding to the largest eigenvalue of this matrix: its norm. We then choose a unitary $u$ transforming this projection $e$ into  $S_1^rS_1'^r(S_1'^*)^r(S_1^*)^r$. Moreover, we implement $\phi$ locally by an isometry  $w\in\mathcal O_n\otimes_\Theta \mathcal O_m$. Using all these elements in $\mathcal O_n\otimes_\Theta \mathcal O_m$, we can define $z\in A$, such that $zyz^*=1$. This implies that $zx^*xz^*$ is  invertible and we finish the proof by putting $a:= b^*x^*$ and $b:=z^*(zx^*xz^*)^{-\frac{1}{2}}$. Then $axb=1$. 

Let us now work out the details.
We define the following subsets of $\mathcal O_n\otimes_\Theta \mathcal O_m$.
\begin{itemize}
\item $\mathcal F_{k,l}:=\textnormal{span}\big\{s_\mu t_{\mu'} t^*_{\nu'}s^*_{\nu} \, \big| \, \lvert\mu\rvert=\lvert\nu\rvert=k, \, \lvert\mu'\rvert=\lvert\nu'\rvert=l\big\}$ \; for $k,l\in\N_0$
\item $\mathcal F_{\N_0,\N_0}:=\overline{\textnormal{span}}\big\{s_\mu t_{\mu'} t^*_{\nu'}s^*_{\nu} \, \big| \, \lvert\mu\rvert=\lvert\nu\rvert, \lvert\mu'\rvert=\lvert\nu'\rvert\big\}=\overline{\bigcup_{k,l\in\N_0}\mathcal F_{k,l}}$ 
\item $\mathcal F_{\N_0,\bullet}:=\overline{\textnormal{span}}\big\{s_\mu t_{\mu'} t^*_{\nu'}s^*_{\nu}\, \big| \, \lvert\mu\rvert=\lvert\nu\rvert\big\}$ 
\item $\mathcal F_{\bullet,\N_0}:=\overline{\textnormal{span}}\big\{s_\mu t_{\mu'} t^*_{\nu'}s^*_{\nu} \, \big| \, \lvert\mu'\rvert=\lvert\nu'\rvert\big\}$ 
\end{itemize}

\begin{lemma}\label{Fabsorbessmallerones}
The above subsets $\mathcal F_{k,l}$,  $\mathcal F_{\N_0,\N_0}$,  $\mathcal F_{\N_0,\bullet}$,  $\mathcal F_{\bullet,\N_0}$ are $C^*$-subalgebras of $\mathcal O_n\otimes_\Theta \mathcal O_m$ and we have $\mathcal F_{k,l}=\cup_{i\leq k, j\leq l} \mathcal F_{i,j}$. Actually, $\mathcal F_{k,l}$ is isomorphic to $M_{n^km^l}(\mathbb C)$.
\end{lemma}
\begin{proof}
Using the well-known relations $t_\mu^*t_\nu=\delta_{\mu\nu}$ for multi indices $\lvert\mu\rvert=\lvert\nu\rvert$ and similarly for $s_{\mu'}$, we infer that all these subsets are in fact $^*$-subalgebras. Moreover, all of these $^*$-subalgebras are closed. For $\mathcal F_{k,l}$ this follows from the fact that the elements  $t_\mu s_{\mu'} s^*_{\nu'}t^*_{\nu} $ satisfy the relations of  matrix units and hence $\mathcal F_{k,l}$ is isomorphic to $M_{n^km^l}(\mathbb C)$.
Finally, let $s_{\mu}t_{\mu'}t^*_{\nu'}s^*_{\nu}\in\mathcal F_{i,j}$ for $i\leq k, j\leq l$ and write $1$ as the sum $P$ of all projections $s_{\delta} t_{\epsilon'}t^*_{\epsilon'}s^*_{\delta}$ with $\lvert\delta\rvert=k-i$ and $\lvert\epsilon'\rvert=l-j$. Then $s_{\mu}t_{\mu'}t^*_{\nu'}s^*_{\nu}=s_{\mu}t_{\mu'}Pt^*_{\nu'}s^*_{\nu}\in\mathcal F_{k,l}$.
\end{proof}

Next we are going to construct a faithful expectation for $\mathcal O_n\otimes_\Theta \mathcal O_m$, similar to the one for $\mathcal O_n$ or the rotation algebra $A_q$. Recall that for a unital $C^*$-algebra $A$, a unital, linear,  positive map  $\phi:A\to B$ is an expectation, if $B\subset A$ is a $C^*$-subalgebra and $\phi^2=\phi$. It is faithful, if it maps non-zero positive elements to non-zero positive elements.

Let $\rho_{\zeta,\xi}:\mathcal O_n\otimes_\Theta \mathcal O_m\to\mathcal O_n\otimes_\Theta \mathcal O_m$ be the automorphisms mapping $s_i\mapsto\zeta s_i$ and $t_j\mapsto\xi t_j$, where $\zeta, \xi\in\C$ are scalars of absolute value one. Put, for $x\in\mathcal O_n\otimes_\Theta \mathcal O_m$:
\begin{gather*}
\phi_1(x):= \int_0^1\rho_{e^{2\pi it},1}(x)\textnormal{dt} \\
\phi_2(x):= \int_0^1\rho_{1,e^{2\pi it}}(x)\textnormal{dt}
\end{gather*}

\begin{lemma} \label{ExpectOmn}
The maps $\phi_1:\mathcal O_n\otimes_\Theta \mathcal O_m\to \mathcal F_{\N_0,\bullet}$,  $\phi_2:\mathcal O_n\otimes_\Theta \mathcal O_m\to\mathcal F_{\bullet,\N_0}$ and $\phi:=\phi_1\circ\phi_2=\phi_2\circ\phi_1:\mathcal O_n\otimes_\Theta \mathcal O_m\to\mathcal F_{\N_0,\N_0}$ are faithful expectations.
\end{lemma}
\begin{proof}
 The map $(\zeta,\xi)\mapsto \rho_{\zeta,\xi}(x)$ is continuous in norm, for all  $x\in\mathcal O_n\otimes_\Theta \mathcal O_m$, as $\rho_{\zeta,\xi}(s_\mu t_{\mu'} t^*_{\nu'}s^*_{\nu})=\zeta^{\lvert\mu\rvert - \lvert\nu\rvert}\xi^{\lvert\mu'\rvert - \lvert\nu'\rvert}s_\mu t_{\mu'} t^*_{\nu'}s^*_{\nu}$. Thus, $\phi_1$ and $\phi_2$ are well-defined as limits of sums $\frac{1}{N}\sum^N_{k=1}\rho_{e^{2\pi i\theta_k},1}(x)$ for partitions $\theta_1,\ldots,\theta_N$ of the unit interval (likewise for $\phi_2$). They are unital, linear, positive and faithful as may be deduced immediately from this approximation by finite sums.

Applying $\phi_1$ to $s_\mu t_{\mu'} t^*_{\nu'}s^*_{\nu}$ (where $\mu,\mu',\nu$ and $\nu'$ are of arbitrary length) yields:
\begin{align*}
\phi_1(s_\mu t_{\mu'} t^*_{\nu'}s^*_{\nu})
&=\left(\int^1_0e^{2\pi i(\lvert\mu\rvert -\lvert\nu\rvert)t}\textnormal{dt}\right)s_\mu t_{\mu'} t^*_{\nu'}s^*_{\nu}\\
&=\begin{cases} s_\mu t_{\mu'} t^*_{\nu'}s^*_{\nu}  & \textnormal{for } \lvert\mu\rvert=\lvert\nu\rvert \\ 
0& \textnormal{otherwise} \end{cases}
\end{align*}
Therefore, since the span of the words $s_\mu t_{\mu'}t'^*_{\nu'}s^*_{\nu}$ is dense in $\mathcal O_n\otimes_\Theta \mathcal O_m$, we get $\phi_1^2=\phi_1$ and the image of $\phi_1$ is $\mathcal F_{\N_0,\bullet}$, similarly for $\phi_2$.

The composition of two faithful expectations is again a unital, linear, positive and faithful map. If they commute, we even get an expectation. This is the case for $\phi_1$ and $\phi_2$, due to the following computation.
\[\phi_1\circ\phi_2(s_\mu t_{\mu'} t^*_{\nu'}s^*_{\nu})=\begin{cases} s_\mu t_{\mu'} t^*_{\nu'}s^*_{\nu}  & \textnormal{for } \lvert\mu\rvert=\lvert\nu\rvert \textnormal{ and }\lvert\mu'\rvert=\lvert\nu'\rvert\\ 
0& \textnormal{otherwise} \end{cases}\]
The same holds for $\phi_2\circ\phi_1(s_\mu t_{\mu'} t^*_{\nu'}s^*_{\nu})$, thus $\phi_1\circ\phi_2=\phi_2\circ\phi_1$ and we are done.
\end{proof}

Let us now implement $\phi$ locally by an isometry using a standard trick.

\begin{lemma} \label{PhiAlsW}
For all $k,l\in\N_0$, there exists an isometry $w\in\mathcal O_n\otimes_\Theta \mathcal O_m$, which commutes with all elements in $\mathcal F_{k,l}$, such that $\phi(y)=w^*yw$ for all   $y\in\textnormal{span}\left\{s_\mu t_{\mu'}t^*_{\nu'}s^*_\nu \, \big| \, \lvert\mu\rvert,\lvert\nu\rvert\leq k \textnormal{ and }\lvert\mu'\rvert,\lvert\nu'\rvert\leq l\right\}$.
\end{lemma}
\begin{proof} Put $s_\gamma:=s_1^{2k}s_2t^{2l}_1t_2$ and define $w$ by:
\[w:=\sum_{\lvert\delta\rvert=k,\lvert\epsilon'\rvert=l} s_\delta t_{\epsilon'}s_\gamma t^*_{\epsilon'}s^*_\delta \]
Then, $w$ is an isometry.
We have $ws_\mu t_{\mu'}=s_\mu t_{\mu'}s_\gamma$ and $t^*_{\nu'}s^*_\nu w=s_\gamma t^*_{\nu'}s^*_\nu$, for all $\mu, \mu', \nu, \nu'$ with $\lvert\mu\rvert=\lvert\nu\rvert=k$ and $\lvert\mu'\rvert=\lvert\nu'\rvert=l$. We conclude  $ws_\mu t_{\mu'}t^*_{\nu'}s^*_\nu =s_\mu t_{\mu'}t^*_{\nu'}s^*_\nu w$, thus $w$ commutes with $\mathcal F_{k,l}$. 

Now let $y=s_\mu t_{\mu'}t^*_{\nu'}s^*_\nu$ be a word with $\lvert\mu\rvert,\lvert\nu\rvert\leq k$  and $\lvert\mu'\rvert,\lvert\nu'\rvert\leq l$. If $\lvert\mu\rvert=\lvert\nu\rvert$ and  $\lvert\mu'\rvert=\lvert\nu'\rvert$, then $y\in\mathcal F_{k,l}$ by Lemma \ref{Fabsorbessmallerones}, thus we have $w^*yw=w^*wy=y=\phi(y)$.
In the case of $\lvert\mu\rvert\neq\lvert\nu\rvert$ or $\lvert\mu'\rvert\neq\lvert\nu'\rvert$, we have $w^*yw=0=\phi(y)$.
\end{proof}

\begin{proposition}\label{OmnIsPurelyInfinite}
$\mathcal O_n\otimes_\Theta \mathcal O_m$ is purely infinite.
\end{proposition}
\begin{proof}
 Let $0\neq x\in\mathcal O_n\otimes_\Theta \mathcal O_m$. Under the faithful expectation $\phi$, the  non-zero, positive element $x^*x$ is mapped   to the non-zero, positive element $\phi(x^*x)$ of norm $1$, if suitably scaled. Since the linear span $\mathcal S$ of all elements $s_\mu t_{\mu'}t^*_{\nu'}s^*_\nu$, where $\mu, \mu', \nu, \nu'$ are multi indices of arbitrary length, is dense in $\mathcal O_n\otimes_\Theta \mathcal O_m$, there is a self-adjoint element $y\in\mathcal S$, such that $\lVert x^*x-y\rVert < \frac{1}{4}$. We conclude that $\lVert\phi(y)\rVert>\frac{3}{4}$, as 
$1=\lVert\phi(x^*x)\rVert\leq \lVert\phi(x^*x-y)\rVert + \lVert\phi(y)\rVert< \frac{1}{4}+\lVert\phi(y)\rVert$.

Let $r$  be the maximal length of the multi indices of the summands of $y$ from its presentation as an element in $\mathcal S$. By Lemma \ref{PhiAlsW}, there is an isometry $w\in\mathcal O_n\otimes_\Theta \mathcal O_m$ which commutes with $\mathcal F_{r,r}$ such that $\phi(y)=w^*yw$.

By Lemma \ref{Fabsorbessmallerones}, we can view $\phi(y)\in\mathcal F_{r,r}\cong M_{n^rm^r}(\C)$  as a self-adjoint matrix. Thus, there is a minimal projection $e\in\mathcal F_{r,r}$, such that $e\phi(y)=\phi(y)e=\lVert\phi(y)\rVert e > \frac{3}{4}e$. There is also a unitary $u\in\mathcal F_{r,r}$ with $ueu^*=s_1^rt_1^r(t_1^*)^r(s_1^*)^r$, transforming one minimal projection into another.

Put $z:=\lVert\phi(y)\rVert^{-\frac{1}{2}}(t_1^*)^r(s_1^*)^ruew^*$. Then $\lVert z\rVert\leq\lVert\phi(y)\rVert^{-\frac{1}{2}}<\frac{2}{\sqrt{3}}$ and we have:
\[zyz^*=\lVert \phi(y)\rVert^{-1}(t_1^*)^r(s_1^*)^rue\phi(y) eu^*s_1^rt_1^r=(t_1^*)^r(s_1^*)^rueu^*s_1^rt_1^r=1\]
Hence $zx^*xz^*$ is invertible:
\[\lVert 1-zx^*xz^*\rVert=\lVert z(y-x^*x)z^*\rVert < \frac{4}{3} \cdot \frac{1}{4} = \frac{1}{3} < 1\]
Finally put $a:= b^*x^*$ and $b:=z^*(zx^*xz^*)^{-\frac{1}{2}}$. Then $axb=1$.
\end{proof}

\subsection{$K$-groups of $\mathcal O_n\otimes_\Theta \mathcal O_m$}

Finally, we  calculate the $K$-groups of $\mathcal O_n\otimes_\Theta \mathcal O_m$ building on the work by Cuntz \cite{Cuntz81}. 
Recall, that Cuntz gave a slightly different definition for the twist of a $C^*$-algebra $A$ with $\mathcal O_m$ in his article: Let $A\subset B(H)$ be a unital $C^*$-algebra, $\alpha_1,\ldots,\alpha_m$ be pairwise commuting automorphisms of $A$ and let  $u_1,\ldots,u_m\in B(H)$ be  pairwise commuting unitaries implementing the automorphisms by $\alpha_i=\textnormal{Ad}(u_i)$. We denote by $\mathcal{U}:=(u_1,\ldots,u_m)$ the tuple of the unitaries and by $A \times_\mathcal{U} \mathcal O_m$ the $C^*$-subalgebra of $B(H)\otimes\mathcal O_m$ generated by all elements $a\otimes 1$ for $a\in A$ together with $u_1\otimes S_1$, \dots, $u_m\otimes S_m$. 

\begin{proposition}[{\cite[th. 1.5]{Cuntz81}}] \label{CuntzKTheorie}
In the situation as above,  the following 6-term  sequence is exact.
\begin{align*}
 &&K_0(A) &&\stackrel{\textnormal{id}-\sum\alpha_{i*}^{-1}}{\longrightarrow} &&K_0(A) &&\longrightarrow &&K_0(A \times_\mathcal{U} \mathcal O_m)\\
 &&\uparrow &&                               &&         &&    &&\downarrow\\
 &&K_1(A \times_\mathcal{U} \mathcal O_m)&&\longleftarrow &&K_1(A) &&\stackrel{\textnormal{id}-\sum\alpha_{i*}^{-1}}{\longleftarrow} &&K_1(A) 
\end{align*}
Here $K_{j}(A)\to K_j(A \times_\mathcal{U} \mathcal O_m)$ for $j=0,1$ is the map induced by the natural inclusion of $A$ into $A \times_\mathcal{U} \mathcal O_m$ via $a\mapsto a\otimes 1$.
\end{proposition}

We now get an according 6-term exact sequence for $\mathcal O_n\otimes_\Theta \mathcal O_m$.

\begin{corollary}\label{OmnKTheoryIndependentOfTheta}
The following sequence in $K$-theory is exact.
\[0\to K_1(\mathcal O_n\otimes_\Theta \mathcal O_m)\to K_0(\mathcal O_n) \stackrel{(m-1)}{\to}  K_0(\mathcal O_n) \to K_0(\mathcal O_n\otimes_\Theta \mathcal O_m)\to 0\]
Hence, the $K$-groups of $\mathcal O_n\otimes_\Theta \mathcal O_m$ are independent from the parameter $\Theta$.
\end{corollary}
\begin{proof}
In our situation, we put $A:=\mathcal O_n$ and $\alpha_j(s_i)=\bar q_{ij}s_i$. We may represent $\mathcal O_n\rtimes_{\alpha_1}\mathbb Z\rtimes_{\alpha_2}\mathbb Z\ldots \rtimes_{\alpha_m}\mathbb Z$ concretely on some Hilbert space by some representation $\pi$ which yields unitaries $u_1,\ldots,u_m$ implementing the automorphisms $\alpha_1,\ldots,\alpha_m$. We may thus form $\mathcal O_n\times_\mathcal U\mathcal O_m$ in the sense of Cuntz and we obtain a $^*$-homomorphism $\sigma:\mathcal O_n\otimes_\Theta \mathcal O_m\to\mathcal O_n\times_\mathcal U\mathcal O_m$ mapping $s_i\mapsto\pi(s_i)\otimes 1$ and $t_j\mapsto  u_j\otimes S_j$. As $\mathcal O_n\otimes_\Theta \mathcal O_m$ is purely infinite, it is simple in particular. Hence, $\sigma$ is an isomorphism.

We now apply Proposition \ref{CuntzKTheorie} to $\mathcal O_n\otimes_\Theta \mathcal O_m\cong \mathcal O_n \times_{\mathcal U} \mathcal O_m$. Since the automorphisms $\alpha_i$  are homotopic to the identity, we get $\alpha_{i*}=\id$ for all $i=1,\ldots,m$ on the level of $K$-theory. Thus the map $\id-\sum\alpha_{i*}^{-1}$ is  multiplication by $(m-1)$ on $K_0(\mathcal O_n)=\mathbb Z / (n-1)\mathbb Z$.
As $K_1(\mathcal O_n)=0$, we end up with the  exact sequence of our assertion.

Now, the map $K_0(\mathcal O_n) \stackrel{(m-1)}{\to}  K_0(\mathcal O_n)$ is independent from $\Theta$ and so are its kernel $K_1(\mathcal O_n\otimes_\Theta \mathcal O_m)$ and its image $K_0(\mathcal O_n\otimes_\Theta \mathcal O_m)$.
Note that  the isomorphisms of
$K_0(\mathcal O_n\otimes_\Theta \mathcal O_m)$ and $K_0(\mathcal O_n\otimes_{\Theta'} \mathcal O_m)$ for different parameters $\Theta$ and $\Theta'$ map units to units.
\end{proof}

\begin{remark}
We may use the K\"unneth formula to compute the $K$-theory of ${\mathcal O_m\otimes\mathcal O_n}$ explicitly. Recall that $\mathcal O_m\otimes\mathcal O_n$ satisfies the UCT (see Prop. \ref{OmnIsNuclear}).
The K\"unneth formula for K-theory, see \cite[th. 23.1.3]{Black},  gives the following short exact sequences, $j\in\mathbb Z_2$,
\begin{gather*}
0 \rightarrow \bigoplus_{i\in\mathbb Z_2}  K_i(\mathcal{O}_n) \otimes_\mathbb{Z} K_{i + j}(\mathcal{O}_m) \rightarrow K_j(\mathcal{O}_n \otimes \mathcal{O}_m) 
\\
\rightarrow \bigoplus_{i\in\mathbb Z_2} Tor_1^\mathbb{Z}(K_{i}(\mathcal{O}_n), K_{i + j + 1}(\mathcal{O}_m)) \rightarrow 0
\end{gather*}
Let $d = \gcd(n - 1, m - 1)$. 
It is a well known fact in homological algebra, see \cite{McLane}, that for an abelian group $A$
\[
Tor_1^\mathbb{Z}(A, \mathbb{Z}/d\mathbb{Z}) \simeq Ann_A(d) = \{a \in A  \mid  da = 0 \}.
\]
In particular,
\[
Tor_1^\mathbb{Z}(\mathbb{Z}/n\mathbb{Z}, \mathbb{Z}/m\mathbb{Z}) \simeq \mathbb{Z}/\gcd(n,m)\mathbb{Z}.
\]
Recall that, see \cite{Cuntz_Ktheory},
\[
K_0(\mathcal{O}_n) = \mathbb{Z} / (n - 1) \mathbb{Z}, \quad K_1(\mathcal{O}_n) = 0.
\]
Hence, for $\mathcal{O}_n \otimes \mathcal{O}_m$, one has the following short exact sequences:
\begin{gather*}
0 \rightarrow \mathbb{Z}/(n-1)\mathbb{Z} \otimes_\mathbb{Z} \mathbb{Z}/(m-1)\mathbb{Z} \rightarrow K_0(\mathcal{O}_n \otimes \mathcal{O}_m) \rightarrow 0 \rightarrow 0,
\\
0 \rightarrow 0 \rightarrow K_1(\mathcal{O}_n \otimes \mathcal{O}_m) \rightarrow \mathbb{Z} / d \mathbb{Z} \rightarrow 0.\qedhere
\end{gather*}
This implies
\[ 
K_0(\mathcal{O}_n \otimes \mathcal{O}_m) \simeq \mathbb{Z} / d\mathbb{Z}, \quad  K_1(\mathcal{O}_n \otimes\mathcal{O}_m) \simeq \mathbb{Z} / d\mathbb{Z}.
\]
\end{remark}

\subsection{$\mathcal O_n\otimes_\Theta \mathcal O_m$ is isomorphic to $\mathcal O_n\otimes \mathcal O_m$}

We may now put together all ingredients in order to apply Kirchberg's Theorem \cite{Kirchberg} which is as follows.

\begin{proposition} \label{KirchbThm}
 Let $A$ and $B$ be unital, separable, nuclear, simple and purely infinite $C^*$-algebras in the bootstrap class $\mathcal N$ with $K_j(A)\cong K_j(B)$ for $j=0,1$ (with matching units in the case of $j=0$). Then $A\cong B$. 
\end{proposition} 

\begin{theorem}
For any parameter $\Theta$, the $C^*$-algebra $\mathcal O_n\otimes_\Theta \mathcal O_m$ is isomorphic to $\mathcal O_m\otimes\mathcal O_n$. Hence we may not twist the tensor product of two Cuntz algebras in this sense.
\end{theorem}
\begin{proof}
By Proposition \ref{OmnIsNuclear}, $\mathcal O_n\otimes_\Theta \mathcal O_m$ is nuclear and in $\mathcal N$ and by Proposition \ref{OmnIsPurelyInfinite}, it is purely infinite and simple. By Corollary \ref{OmnKTheoryIndependentOfTheta}, the $K$-groups of  $\mathcal O_n\otimes_\Theta \mathcal O_m$ and $\mathcal O_m\otimes\mathcal O_n$ coincide. Thus, by Kirchberg's Theorem, $\mathcal O_n\otimes_\Theta \mathcal O_m$ is isomorphic to $\mathcal O_m\otimes\mathcal O_n$. 
\end{proof}

\begin{remark}
The above result on $\mathcal O_n\otimes_\Theta \mathcal O_m$ has been part of the fourth author's PhD thesis from 2011 and it has not been published in any other article.
\end{remark}

\subsection{ The isomorphism $\mathcal{O}_n \otimes_q \mathcal{O}_m \simeq \mathcal{O}_n \otimes \mathcal{O}_m$. The Rieffel deformation approach}
In this section we return to one parameter case of $\mathcal O_n\otimes_q \mathcal O_m$, and present an approach to the proof of isomorphism
\[
\mathcal{O}_n \otimes_q \mathcal{O}_m \simeq \mathcal{O}_n \otimes \mathcal{O}_m,\quad |q|<1 
\]
which is based on the properties of Rieffel deformation.

In \cite{Skandalis}, the authors have shown that
for every $C^*$-algebra $\mathcal A$ with an action $\alpha$ of $\mathbb{R}$, there exists a KK-isomorphism 
$t_\alpha \in KK_{1}(\mathcal A, \mathcal A \rtimes_\alpha \mathbb{R})$. This $t_\alpha$ is a generalization of the Connes-Thom isomorphisms for K-theory. Below we will denote by 
$\circ : KK(\mathcal A, \mathcal B) \times KK(\mathcal B, \mathcal C) \rightarrow KK(\mathcal A, \mathcal C)$ the Kasparov product, and by $\boxtimes : KK(\mathcal A, \mathcal B) \times KK(\mathcal C, \mathcal D) \rightarrow KK(\mathcal A \otimes \mathcal C, \mathcal B \otimes \mathcal D)$ the exterior tensor product. Given a homomorphism $\phi : \mathcal A \rightarrow\mathcal B$, put $[\phi] \in KK(\mathcal A, \mathcal B)$ to be the induced KK\nobreakdash-morphism. For more details see \cite{Black,Knudsen}.

We list some properties of $t_\alpha$ that will be used below.
\begin{enumerate}
    \item Inverse of $t_\alpha$ is given by $t_{\widehat{\alpha}}$, where $\widehat{\alpha}$ is the dual action.
    \item If $\mathcal A = \mathbb{C}$ with the trivial action of $\mathbb{R}$, then the corresponding element
    \[
    t_1 \in KK_1(\mathbb{C}, C_0(\mathbb{R})) \simeq \mathbb{Z}
    \]
    is the generator of the group.
    \item Let $\phi: (\mathcal A, \alpha) \rightarrow (\mathcal B, \beta)$ be an equivariant homomorphism. Then the following diagram commutes in KK-theory
    \[
    \begin{tikzcd}
    \mathcal A \arrow[r, "t_\alpha"] \arrow[d, "\phi"] & \mathcal A \rtimes_\alpha \mathbb{R} \arrow[d, "\phi \rtimes \mathbb{R}"] \\
    \mathcal B \arrow[r, "t_\beta"] & \mathcal B \rtimes_\beta \mathbb{R}
    \end{tikzcd}
    \]
    \item Let $\beta$ be an action of $\mathbb{R}$ on $\mathcal B$. For the action $\gamma = \id_{\mathcal A} \otimes \beta$ on $\mathcal A \otimes\mathcal B$ we have
    \[
    t_\gamma = \mathbf1_{\mathcal A} \boxtimes\, t_\beta.
    \]
\end{enumerate}

Further we will need the following version of classification result by Kirchberg and Philips:

\begin{theorem}[\cite{Kirchberg}, Corollary 4.2.2]
Let $\mathcal A$ and $\mathcal B$ be separable nuclear unital purely infinite simple $C^*$-algebras, and suppose that there exists an invertible element $\eta \in KK(\mathcal A,\mathcal B)$, such that $[\iota_{\mathcal A}] \circ \eta = [\iota_{\mathcal B}]$, where $\iota_{\mathcal A} : \mathbb{C} \rightarrow \mathcal A$ is defined by  $\iota_{\mathcal A}(1) =\mathbf1_{\mathcal A}$, and $\iota_{\mathcal B} : \mathbb{C} \rightarrow\mathcal B$ is defined by $\iota_{\mathcal B}(1) =\mathbf1_{\mathcal B}$. Then $\mathcal A$ and $\mathcal B$ are isomorphic.
\end{theorem}
Notice that the conditions of the theorem above does not require $C^*$-algebras $\mathcal A$ and $\mathcal B$ to be in the bootstrap class $\mathcal N$.
\begin{theorem}\label{q_stab_res}
The $C^*$-algebras $\mathcal{O}_n \otimes_q \mathcal{O}_m$ and $\mathcal{O}_n \otimes \mathcal{O}_m$ are isomorphic for any $|q| = 1$.
\end{theorem}

\begin{proof}
Throughout the proof we will distinguish between the actions of $\mathbb{T}^2$ on $\mathcal{O}_n \otimes \mathcal{O}_m$ and on $\mathcal{O}_n \otimes_q \mathcal{O}_m$, denoting the latter by $\alpha^q$. As shown above, the both algebras are separable nuclear unital simple and purely infinite.

Further,  
Proposition \ref{prop4}, and the isomorphism $\mathcal O_n\otimes_q\mathcal O_m\simeq (\mathcal O_n\otimes\mathcal O_m)_{\Theta_q}$ yield the isomorphism
\[
\Psi : (\mathcal{O}_n \otimes \mathcal{O}_m) \rtimes_\alpha \mathbb{R}^2 \rightarrow (\mathcal{O}_n \otimes_q \mathcal{O}_m) \rtimes_{\alpha^q} \mathbb{R}^2.
\]
Decompose the crossed products as follows:
\begin{align*}
(\mathcal{O}_n \otimes \mathcal{O}_m) \rtimes_\alpha \mathbb{R}^2 &\simeq (\mathcal{O}_n \otimes \mathcal{O}_m) \rtimes_{\alpha_1} \mathbb{R} \rtimes_{\alpha_2} \mathbb{R},
\\
(\mathcal{O}_n \otimes_q \mathcal{O}_m) \rtimes_{\alpha^q} \mathbb{R}^2 &\simeq (\mathcal{O}_n \otimes_q \mathcal{O}_m) \rtimes_{\alpha_1^q} \mathbb{R} \rtimes_{\alpha_2^q} \mathbb{R}.
\end{align*}
Define
\begin{align*}
t_\alpha &= t_{\alpha_1} \circ (\mathbf1_{C_0(\mathbb{R})} \boxtimes t_{\alpha_2}) \in KK(\mathcal{O}_n \otimes \mathcal{O}_m,(\mathcal{O}_n \otimes \mathcal{O}_m) \rtimes_\alpha \mathbb{R}^2),
\\
t_{\alpha^q} &= t_{\alpha_1^q} \circ  (\mathbf1_{C_0(\mathbb{R})} \boxtimes t_{\alpha_2^q}) \in KK(\mathcal{O}_n \otimes_q \mathcal{O}_m,(\mathcal{O}_n \otimes_q \mathcal{O}_m) \rtimes_{\alpha^q} \mathbb{R}^2),
\end{align*}
Then
\[ \eta = t_{\alpha^q} \circ [\Psi] \circ t_{\alpha}^{-1} \in KK(\mathcal{O}_n \otimes_q \mathcal{O}_m, \mathcal{O}_n \otimes \mathcal{O}_m) \]
is a $KK$-isomorphism. The property $[\iota_{\mathcal{O}_n \otimes_q \mathcal{O}_m}] \circ \eta = [\iota_{\mathcal{O}_n \otimes \mathcal{O}_m}]$ follows from the commutativity of the following diagram
\[
\begin{tikzcd}[column sep = 0.3cm]
\mathbb{C} \arrow[rr, "t_1 \circ (\mathbf1_{C_0(\mathbb{R})} \boxtimes\, t_1)"] \arrow[d, "\iota_{\mathcal{O}_n \otimes_q \mathcal{O}_m}"]
& &
C_0(\mathbb{R}^2) \arrow[rr,"(\mathbf1_{C_0(\mathbb{R})} \boxtimes\, t_1)^{-1} \circ\, t_1^{-1}"] \arrow[dl,"\iota_{\mathcal{O}_n \otimes_q \mathcal{O}_m} \rtimes \mathbb{R}^2"'] \arrow[dr,"\iota_{\mathcal{O}_n \otimes \mathcal{O}_m} \rtimes \mathbb{R}^2"]
& &
\mathbb{C} \arrow[d, "\iota_{\mathcal{O}_n \otimes \mathcal{O}_m}" left] \\
\mathcal{O}_n \otimes_q \mathcal{O}_m \arrow[r,"t_{\alpha^q}"] &
(\mathcal{O}_n \otimes_q \mathcal{O}_m) \rtimes_{\alpha^q} \mathbb{R}^2 \arrow[rr,"\Psi"] & &
(\mathcal{O}_n \otimes \mathcal{O}_m) \rtimes_{\alpha} \mathbb{R}^2 \arrow[r,"t_{\alpha}^{-1}"] &
\mathcal{O}_n \otimes \mathcal{O}_m.\ 
\end{tikzcd}
\]
\end{proof}
\begin{remark}
The proof presented above was obtained by the first author independently of the proof of multi-parameter case.
\end{remark}

\subsection{Computation of $\mathsf{Ext}$ for $\mathcal{E}_{n,m}^q$}

Let us finish with some remarks on $\mathcal{E}_{n,m}^q$. Firstly, the simplicity of $\mathcal O_n\otimes_q\mathcal O_m$ implies that $\mathcal{M}_q\subset\mathcal E_{n,m}^q$ is the largest ideal.
\begin{corollary}\label{mq_unique}
 The ideal $\mathcal{M}_q\subset\mathcal E_{n,m}^q$ is the unique largest ideal.
\end{corollary}

\begin{proof}
 Let $\eta\colon\mathcal E_{n,m}^q\rightarrow \mathcal{O}_n\otimes_q\mathcal{O}_m$ be the quotient homomorphism.
 Suppose that $\mathcal{J}\subset\mathcal{E}_{n,m}^q$ is a two-sided $*$-ideal. Due to the simplicity of
 $\mathcal{O}_n\otimes_q\mathcal{O}_m$ we have that either $\eta(\mathcal J)=\{0\}$ and
 $\mathcal J\subset\mathcal M_q$, or $\eta(\mathcal J)= \mathcal{O}_n\otimes_q\mathcal{O}_m$. In the latter case,
 $\mathbf1+x\in \mathcal J$ for a certain $x\in\mathcal M_q$. For any $0<\varepsilon<1$, choose
 $N_{\varepsilon}\in \mathbb N$,
 such that for
 \[
 x_{\varepsilon}=\mspace{-9mu}
 \sum_{\substack{\varepsilon_1,\varepsilon_2 \in\{0,1\},\\
 \varepsilon_1+\varepsilon_2\ne 0}}
 \sum_{\substack{\mu_1,\mu_2\in\Lambda_n,\\ |\mu_j|\le N_{\varepsilon}}}
 \sum_{\substack{\nu_1,\nu_2\in\Lambda_m,\\ |\nu_j|\le N_{\varepsilon}}} \mspace{-9mu}
 \Psi_{\mu_1,\mu_2\nu_1\nu_2}^{(\varepsilon_1,\varepsilon_2)}s_{\mu_1}t_{\nu_1}
 (\mathbf 1-P)^{\varepsilon_1}(\mathbf 1-Q)^{\varepsilon_2}t_{\nu_2}^*s_{\mu_2^*}\in\mathcal M_q
 \]
one has $\|x-x_{\varepsilon}\|<\varepsilon$. Notice that for any $\mu\in\Lambda_n$, $\nu\in\Lambda_m$
with
$|\mu|,|\nu|> N_{\varepsilon}$ one has ${s_{\mu}^* t_{\nu}^* x_{\varepsilon}=0}$.

Fix $\mu\in\Lambda_n$ and $\nu\in\Lambda_m$, $|\mu|=|\nu|>N_{\varepsilon}$, then
\[
y_{\varepsilon}=s_{\mu}^*t_{\nu}^*(\mathbf 1- x)t_{\nu}s_{\mu}=\mathbf 1 - s_{\mu}^*
t_{\nu}^*(x- x_{\varepsilon})t_{\nu}s_{\mu}
\in\mathcal J.
\]
Thus $\|s_{\mu}^*t_{\nu}^*(x- x_{\varepsilon})t_{\nu}s_{\mu}\|<\varepsilon$ implies that $y_{\varepsilon}$ is
invertible,
so $\mathbf 1\in\mathcal J$.
 \end{proof}

Secondly, we may show that $\mathsf{Ext}(\mathcal{O}_n \otimes_q \mathcal{O}_m, \mathcal{M}_q)=0$ if $\gcd(n-1,m-1)=1$. To this end we compute first
the K-theory of $\mathcal{M}_q$.

\begin{theorem}
Let $d = \gcd(n - 1, m - 1)$. Then
\[
K_0(\mathcal{M}_q) \simeq \mathbb{Z} / d\mathbb{Z} \oplus \mathbb{Z}, \text{ }  K_1(\mathcal{M}_q) \simeq 0.
\]
\end{theorem}

\begin{proof}
The isomorphism $\mathcal E_{n,m}^q\simeq (\mathcal O_n^{(o)}\otimes\mathcal O_m^{(0)})_{\Theta_q}$, Proposition \ref{Rieff_K_theory}, and \cite{Cuntz_Ktheory}, Proposition 3.9, imply that
\begin{align*}
K_0(\mathcal{E}_{n,m}^q) &= K_0((\mathcal{O}_n^{(0)}\otimes\mathcal{O}_m^{(0)})_{\Theta_q}) = K_0(\mathcal{O}_n^0 \otimes \mathcal{O}_m^0) = \mathbb{Z},
\\
K_1(\mathcal{E}_{n,m}^q) &= K_1((\mathcal{O}_n^{(0)}\otimes\mathcal{O}_m^{(0)})_{\Theta_q}) = K_1(\mathcal{O}_n^0 \otimes \mathcal{O}_m^0) = 0.
\end{align*}
Applying the 6-term exact sequence for
\[ 0
\rightarrow \mathbb{K} \rightarrow \mathcal{M}_q \rightarrow \mathcal O_n\otimes\mathbb K \oplus \mathcal O_m\otimes\mathbb K \rightarrow 0 ,
\]
we get
\[
\begin{tikzcd}
\mathbb{Z} \arrow[r] & K_0(\mathcal{M}_q) \arrow[r] & \mathbb{Z}/(n-1)\mathbb{Z} \oplus \mathbb{Z}/(m-1)\mathbb{Z} \arrow[d] \\
0 \arrow[u] & \arrow[l] K_1(\mathcal{M}_q) & \arrow[l] 0
\end{tikzcd}
\]
Then  $K_1(\mathcal{M}_q) = 0$,
and elementary properties of finitely generated abelian groups imply that
\[ K_0(\mathcal{M}_q) = \mathbb{Z} \oplus \mathsf{Tors},\] where $\mathsf{Tors}$ is a direct sum of finite cyclic groups.

Further, the following exact sequence
\[
0 \longrightarrow \mathcal{M}_q \longrightarrow \mathcal{E}_{n,m}^q \rightarrow \mathcal{O}_n \otimes_q \mathcal{O}_m\longrightarrow 0
\]
gives
\[
\begin{tikzcd}
K_0(\mathcal{M}_q) \arrow[r, "p"] & \mathbb{Z} \arrow[r] & \mathbb{Z} / d \mathbb{Z} \arrow[d] \\
\mathbb{Z} / d \mathbb{Z} \arrow[u, "i"] & \arrow[l] 0 & \arrow[l] 0
\end{tikzcd}
\]
The map $p : K_0(\mathcal{M}_q) \simeq \mathbb{Z} \oplus \mathsf{Tors} \rightarrow \mathbb{Z}$ has form $p=(p_1,p_2)$, where
\[
p_1\colon\mathbb Z\rightarrow\mathbb Z,\quad p_2\colon\mathsf{Tors}
\rightarrow\mathbb Z.
\]
Evidently, $p_2=0$, and $p\ne 0$ implies that $\ker p_1=\{0\}$.
Thus,
\[
\ker p = \mathsf{Tors} = \mathrm{Im}(i) \simeq \mathbb{Z} / d\mathbb{Z}.
\qedhere
\]
\end{proof}

\begin{theorem}\label{Ext_comp}
Let $d=\gcd(n - 1, m - 1)=1$. Then $\mathsf{Ext}(\mathcal{O}_n \otimes_q \mathcal{O}_m, \mathcal{M}_q)=0$.
\end{theorem}

\begin{proof}
Recall that for nuclear $C^*$-algebras $\mathsf{Ext}(\mathcal A,\mathcal B) \simeq KK_1(\mathcal A,\mathcal B)$. 

We use  the UCT sequence 
\[ 
0 \rightarrow \bigoplus_{i \in \mathbb{Z}_2} Ext_\mathbb{Z}^1(K_i(\mathcal A), K_i(\mathcal B)) \rightarrow KK_1(\mathcal A, \mathcal B) \rightarrow \bigoplus_{i \in \mathbb{Z}_2} {Hom}(K_i(\mathcal A), K_{i + 1}(\mathcal B)).
\]
for $\mathcal A = \mathcal{O}_n \otimes_q \mathcal{O}_m$ and $\mathcal B = \mathcal{M}_q$.

Since $K_0(\mathcal A)=K_1(\mathcal A)=\mathbb Z/d\mathbb Z$ and $K_0(\mathcal B)=\mathbb Z\oplus\mathbb Z/d\mathbb Z$, $K_1(\mathcal B)=0$, one has
\[ 
{Hom}(K_0(\mathcal A), K_1(\mathcal B)) = 0, \quad {Hom}(K_1(\mathcal A), K_0(\mathcal B)) = \mathbb{Z} / d \mathbb{Z},
\]
and, see \cite{McLane},
\[ 
Ext_\mathbb{Z}^1(K_0(\mathcal A), K_0(\mathcal B)) = \mathbb{Z} / d \mathbb{Z} \oplus \mathbb{Z} / d \mathbb{Z},\quad \ Ext_\mathbb{Z}^1(K_1(\mathcal A), K_1(\mathcal B)) = 0.
\] 
Hence the following sequence is exact
\[ 0
\rightarrow \mathbb{Z} / d \mathbb{Z} \oplus \mathbb{Z} / d \mathbb{Z} \rightarrow KK_1(\mathcal{O}_n \otimes_q \mathcal{O}_m, \mathcal{M}_q) \rightarrow \mathbb{Z} / d \mathbb{Z} \rightarrow 0. 
\]
\end{proof}

By Theorem \ref{Ext_comp}, for the case of $\gcd(n-1, m-1)=1$ one can immediately deduce that extension classes of
\[ 0
\rightarrow \mathcal{M}_q \rightarrow \mathcal{E}_{n,m}^q \rightarrow \mathcal{O}_n \otimes_q \mathcal{O}_m \rightarrow 0,
\]
and
\[ 0
\rightarrow \mathcal{M}_1 \rightarrow \mathcal{E}_{n,m}^1 \rightarrow \mathcal{O}_n \otimes \mathcal{O}_m \rightarrow 0,
\]
coincide in $\mathsf{Ext}(\mathcal{O}_n \otimes \mathcal{O}_m, \mathcal{M}_1)$ and are trivial. These extensions are essential, however in general case one does not have an immediate generalization of Proposition \ref{Voiculescu}. Thus the study of the problem whether $\mathcal{E}_{n,m}^q \simeq \mathcal{E}_{n,m}^1$ would require further investigations, see \cite{E_theory,Eilers}.

\section*{Acknowledgments}

The work on the paper was initiated during the visit  of V. Ostrovskyi, D. Proskurin and R. Yakymiv to Chalmers University of Technology. We appreciate the working atmosphere and stimulating discussions with Prof. Lyudmila Turowska and Prof. Magnus Goffeng. We are grateful to Prof. Pawel Kasprzak for discussions on properties of Rieffel's deformation  and to  Prof. Eugene Lytvynov for  discussions on generalized statistics and his kind hospitality during the visit of D. Proskurin to Swansea University.
We also indebted to Prof. K. Iusenko for helpful comments and remarks.
M.~Weber has been supported by \emph{SFB-TRR 195} and the \emph{DFG}.

\end{document}